\documentclass[12pt]{article}
\usepackage[utf8]{inputenc}
\usepackage{xcolor,soul}
\usepackage{amsmath,amssymb,amsthm,enumerate,graphicx}
\usepackage{mathtools}
\usepackage{ytableau}
\usepackage{tikz,tikz-3dplot}
\usetikzlibrary{decorations.pathreplacing,calligraphy}
\usepackage{soul}

\usepackage{verbatim}
\usepackage{overpic}
\usepackage{geometry}

\usepackage{graphicx}

\theoremstyle{plain}
\newtheorem{theorem}{Theorem}
\newtheorem{proposition}[theorem]{Proposition}
\newtheorem{lemma}[theorem]{Lemma}

\theoremstyle{definition}
\newtheorem{remark}[theorem]{Remark}

\newtheorem{definition}[theorem]{Definition}

\newcommand{\match}{\mathcal{M}}

\definecolor{green}{RGB}{34, 139, 34}

\definecolor{darkpink}{RGB}{245,75,200}

\definecolor{uw_purple}{RGB}{95,45,156} 

\title{Perfect Matching Complexes of Honeycomb Graphs}
\author{Margaret Bayer\\Department of Mathematics\\University of Kansas\\Lawrence, Kansas, U.S.A.\\bayer@ku.edu \and Marija Jeli\'{c} Milutinovi\'{c}\\Faculty of Mathematics\\University of Belgrade\\ Belgrade, Serbia\\marijaj@matf.bg.ac.rs \and  Julianne Vega\thanks{This article is based on work supported by the National Science Foundation under Grant No.\ DMS-1440140 while the authors participated in the 2020/2021 Summer Research in Mathematics Program of the Mathematical Sciences Research Institute, Berkeley, California. M.\ Jeli\'{c} Milutinovi\'{c} has been supported by the Project No.\ 7744592 MEGIC
”Integrability and Extremal Problems in Mechanics, Geometry and
Combinatorics” of the Science Fund of Serbia, and by the Faculty of Mathematics University of Belgrade through the grant (No.\ 451-03-68/2022-14/200104) by the Ministry of
Education, Science, and Technological Development of the Republic of Serbia.}
\\Department of Mathematics\\Maret School\\Washington, D.C. U.S.A.\\jvega@maret.org}

\begin{document}
\maketitle
\begin{abstract}
The {\em perfect matching complex} of a graph is the simplicial complex on the edge set of the graph with facets corresponding to perfect matchings of the graph.  This paper studies the perfect matching complexes, $\match_p(H_{k \times m\times n})$, of honeycomb graphs.  For $k = 1$, $\match_p(H_{1\times m\times n})$ is contractible unless $n\ge m=2$, in which case it is homotopy equivalent to the $(n-1)$-sphere.  Also, $\match_p(H_{2\times 2\times 2})$ is homotopy equivalent to the wedge of two 3-spheres.  The proofs use discrete Morse theory.
\end{abstract}


\section{Introduction and Background}
There has been a great deal of research on the topology of simplicial complexes associated
with graphs.
Much of the early work in this area concerned matching complexes of
complete graphs and of complete bipartite graphs, called ``the
matching complex'' and ``chessboard complexes,'' respectively. See the survey \cite{wachs} by Wachs, and references therein. 
While some other complexes associated with graphs were studied around the same
time (e.g., \cite{kozlov-directed}), much activity was initiated by the
dissertation and subsequent book of Jonsson \cite{jonsson-book}.  In particular, there has
been much work on independence complexes of graphs
(for example, \cite{adamaszek-split,adamaszek-stacho,
ehrenborg,engstrom,matsushita2022}).
The more recent work on matching complexes of various types of graphs includes line tilings ~\cite{bjv-line,matsushita2019}, grid graphs~\cite{Braun_Hough,goyal2021matching,matsushita2018}, and more~\cite{bayer,JelicEtAl,MR2475001,nikseresht,Vega}.
One particular type of graph is the
honeycomb graph (a planar graph that is a tiling of hexagons), studied, for
example, in \cite{JelicEtAl,matsushita2019}.
In this paper we focus on honeycomb graphs
and on a subcomplex of the matching complex, generated by
the faces corresponding to perfect matchings.

\subsection{Perfect matching complexes}


A {\em matching} of a simple graph $G$ is a set
of edges of $G$, no two of which share a vertex.
The {\em matching complex} of $G$, denoted by  $\match(G)$, is the 
simplicial complex whose vertex set is the set of edges of $G$, and whose 
facets (maximal faces) correspond to maximal matchings of $G$. 
\begin{definition}\label{perfect_matching_complex}
 Let $G$ be a simple graph.  A {\em perfect matching} of $G$ is a matching that covers all vertices of the graph.
The {\em perfect matching complex} of $G$, denoted by  $\match_p(G)$, is the 
subcomplex of the matching complex whose 
facets correspond to perfect matchings of $G$. 
\end{definition}

Thus, if $\sigma = \{e_1, \ldots, e_k\}$ is an arbitrary
subset of the edges $E(G)$, $\sigma$ is a simplex (face) in the complex  
$\match_p(G)$ if and only if there is a perfect matching $P$ of $G$ such that 
$\sigma \subseteq P$.  

If the graph $G$ admits at least one perfect matching, the perfect matching complex $\match_p(G)$ is a full-dimensional subcomplex of the matching complex 
$\match(G)$. Otherwise, if there is no perfect matching of the graph $G$, we consider the perfect matching complex to be void.
Note that even finding the number of perfect matchings of a graph is not easy.  The number of perfect matchings of a graph can be expressed in terms of permanents of associated matrices \cite{li}. It is known that computing the permanent is $\# P$-complete \cite{valiant}.

 It turns out that the perfect matching complexes of complete graphs, complete bipartite graphs, paths and cycles are either their entire matching complexes, or complexes that can easily be  determined. Whenever $G$ has an odd number of vertices, the perfect matching complex is void, so we focus only on graphs with even number of vertices.

\begin{itemize}
    \item {\em Complete graph} $G=K_{2n}$. Every matching is a subset of at least one perfect matching, so  $\match_p(K_{2n}) = \match(K_{2n}).$

    \item {\em Complete bipartite graph} $G = K_{m,n}$. If $m \neq n$, $G$ does not contain a perfect matching, so $\match_p(K_{m,n})$ is the void complex. When $m=n$, every matching is a subset of a perfect matching, and $\match_p(K_{n,n}) =\match(K_{n, n}).$
    \item {\em Path} $G = P_{2n}$. If we denote the vertices of the path by $a_1, a_2, \ldots, a_{2n}$ respectively, there is only one perfect matching on $G$, containing edges
    $\{a_{2i-1}, a_{2i}\},$ $i \in \{1, \ldots,  n\}.$ Therefore $\match_p(P_{2n})$ is a simplex $\Delta^{n-1}$ on $n$ vertices.
    \item {\em Cycle} $G = C_{2n}$. There are exactly two perfect matchings on $C_{2n},$ and they are disjoint, so the complex $\match_p(C_{2n})$ is the disjoint union of two simplices of dimension $n-1$, and $\match_p(C_{2n})$ is homotopy equivalent to the 0-sphere $S^0.$

\end{itemize}

\subsection{Bijection between perfect matchings of honeycombs and plane partitions}


A \textbf{honeycomb graph} $H = H_{k \times m \times n}$ is a hexagonal tiling whose congruent, opposite sides are of length $k, m$ and $n$ hexagons.
In his exploration of the topology of matching complexes \cite{jonsson-book}, Jonsson suggests honeycomb graphs to be of interest for further study. In addition, perfect matchings in honeycomb graphs are of interest in chemistry \cite{klein}.  In 2019, discrete Morse theory was used to determine the connectedness bounds of $d$-dimensional faces of matching complexes of $1 \times 1 \times n $ and $2 \times 1 \times n$ honeycomb graphs for $n \geq 1$ \cite{JelicEtAl}. It was later shown by Matsushita that the homotopy type for matching complexes of $1 \times 2 \times n$ honeycomb graphs is a wedge of spheres \cite{matsushita2019}. Beyond the line of hexagons, the homotopy types of matching complexes of honeycomb graphs have been quite elusive. In this paper, we will consider the perfect matching complex of a honeycomb graph and prove that this subcomplex for $H_{1 \times m \times n}$ is contractible or homotopy equivalent to a sphere using the bijection between perfect matchings on honeycomb graphs and plane partitions. 

A \textbf{plane partition} is a two dimensional array of integers that are non-increasing moving from left to right and top to bottom. We define a plane partition $P_{k \times m \times n}$ through a $k \times m$ matrix whose entries are less than or equal to $n$ and follow the non-increasing conditions. A plane partition can be visualized as a pile of unit cubes in the positive octant of $\mathbb{R}^3$ following the non-increasing conditions. 
The perfect matchings of a honeycomb graph are in bijection with the rhombus tilings of a hexagonal region of equilateral triangles, which are in bijection with plane partitions. 
See, for example, \cite{kuperberg}, which gives the number of plane partitions (and hence of perfect matchings) of honeycomb graphs under various symmetry groups.
For an example see Figure~\ref{fig:honeycomb_partition}. 
on page~\pageref{figure4}.
We will use this well-known bijection to determine the homotopy type of 
the perfect matching complexes of $H_{1 \times m \times n}$ and 
$H_{2\times 2\times 2}$.





\section{Hexagonal line tiling}

In this section we use the
nerve theorem, an important theorem in topology, to find the homotopy type of the perfect matching complex of a line of hexagons. For an exposition in the combinatorial context of
simplicial complexes, see \cite{bjorner-topol}.
\begin{definition}
The {\em nerve} of a family of sets $(A_i)_{i\in I}$ is the simplicial complex
${\cal N}(A_i)$ with vertex set $I$ and $\sigma\subseteq I$ a face of 
${\cal N}(A_i)$ if and only if $\bigcap_{i\in\sigma} A_i \ne \emptyset$.
\end{definition}
\begin{theorem}
Suppose $\Delta$ is a simplicial complex, and $(\Delta_i)_{i\in I}$ is a family
of subcomplexes such that $\Delta = \bigcup_{i\in I} \Delta_i$.
If every nonempty finite intersection 
$\bigcap_{i\in J} \Delta_i$ ($J\subseteq I$) is contractible, then 
$\Delta$ and the nerve ${\cal N}(\Delta_i)$ are homotopy equivalent.
\end{theorem}

\begin{theorem}
Let $H_n = H_{1\times 1 \times n}$ be the graph of a line of $n$ hexagons.  
Let $\match_p(H_n)$ be the perfect matching complex of $H_n$.
For $n\ge 2$,  $\match_p(H_n)$ is contractible.
\end{theorem}
\begin{proof}
Let $n\ge 2$.
Label the graph $H_n$  as in  Figure~\ref{hexagon_line}.

\begin{figure}
\begin{center}
\begin{tikzpicture}
\draw[line width = 0.5 pt] (0,2)--(1,3); \node[font=\scriptsize] at (0.25,2.5) {$a_{1,1}$};
\draw[line width = 0.5 pt] (1,3)--(2,2); \node[font=\scriptsize] at (1.25,2.5) {$d_{1,1}$};
\draw[line width = 0.5 pt]  (2,2)--(2,1); \node[font=\scriptsize] at (1.75, 1.5){$b_{1,1}$};
\node[] at (1, 1.25){$T_{1,1}$};
\draw[line width = 0.5 pt]  (2,1)--(1,0); \node[font=\scriptsize] at (1.25, 0.5) {$a_{1,0}$};
\draw[line width = 0.5 pt]  (1,0)--(0,1); \node[font=\scriptsize] at (0.25, 0.5) {$d_{0,0}$};
\draw[line width = 0.5 pt]  (0,1)--(0,2); \node[font=\scriptsize] at (-0.25, 1.5) {$b_{0,1}$};
\draw[line width = 0.5 pt] (2,2)--(3,3); \node[font=\scriptsize] at (2.25, 2.5) {$a_{2,1}$};
\draw[line width = 0.5 pt]  (3,3)--(4,2); \node[font=\scriptsize] at (3.25, 2.5) {$d_{2,1}$};
\draw[line width = 0.5 pt] (4,2)--(4,1); \node[font=\scriptsize] at (3.75, 1.5) {$b_{2,1}$};
\node[] at (3, 1.25){$T_{2,1}$};
\draw[line width = 0.5 pt] (4,1)--(3,0); \node[font=\scriptsize] at (3.25, 0.5) {$a_{2,0}$};
\draw[line width = 0.5 pt] (3,0)--(2,1); \node[font=\scriptsize] at (2.25, 0.5) {$d_{1,0}$};
\draw[line width = 0.5 pt] (10,2)--(11,3); 
\node[font=\scriptsize] at (10,2.5) {$a_{n-1,1}$};
\draw[line width = 0.5 pt] (11,3)--(12,2); \node[font=\scriptsize] at (11.1, 2.5) {$d_{n-1,1}$};
\draw[line width = 0.5 pt] (12,2)--(12,1); \node[font=\scriptsize] at (11.6, 1.5) {$b_{n-1,1}$};
\node[] at (11, 1.25){$T_{n-1,1}$};
\draw[line width = 0.5 pt] (12,1)--(11,0); \node[font=\scriptsize] at (11.1, 0.5) {$a_{n-1,0}$};
\draw[line width = 0.5 pt] (11,0)--(10,1); \node[font=\scriptsize] at (10.1, 0.5) {$d_{n-2,0}$};
\draw[line width = 0.5 pt] (10,1)--(10,2); \node[font=\scriptsize] at (9.6, 1.5) {$b_{n-2,1}$};
\draw[line width = 0.5 pt] (12,2)--(13 ,3); \node[font=\scriptsize] at (12.25, 2.5) {$a_{n,1}$};
\draw[line width = 0.5 pt] (13,3)--(14,2); \node[font=\scriptsize] at (13.25, 2.5) {$d_{n,1}$};
\draw[line width = 0.5 pt] (14,2)--(14,1); \node[font=\scriptsize] at (13.75, 1.5) {$b_{n,1}$};
\node[] at (13, 1.25){$T_{n,1}$};
\draw[line width = 0.5 pt] (14,1)--(13,0); \node[font=\scriptsize] at (13.75, 0.5) {$a_{n,0}$};
\draw[line width = 0.5 pt] (13,0)--(12,1); \node[font=\scriptsize] at (12.1, 0.5) {$d_{n-1,0}$};
\filldraw[black] (0,1) circle (2pt);
\filldraw[black] (0,2) circle (2pt);
\filldraw[black] (1,0) circle (2pt);
\filldraw[black] (1,3) circle (2pt);
\filldraw[black] (2,1) circle (2pt);
\filldraw[black] (2,2) circle (2pt);
\filldraw[black] (3,0) circle (2pt);
\filldraw[black] (3,3) circle (2pt);
\filldraw[black] (4,1) circle (2pt);
\filldraw[black] (4,2) circle (2pt);
\filldraw[black] (6,1.5) circle (2pt);
\filldraw[black] (7,1.5) circle (2pt);
\filldraw[black] (8,1.5) circle (2pt);
\filldraw[black] (10,1) circle (2pt);
\filldraw[black] (10,2) circle (2pt);
\filldraw[black] (11,0) circle (2pt);
\filldraw[black] (11,3) circle (2pt);
\filldraw[black] (12,1) circle (2pt);
\filldraw[black] (12,2) circle (2pt);
\filldraw[black] (13,0) circle (2pt);
\filldraw[black] (13,3) circle (2pt);
\filldraw[black] (14,1) circle (2pt);
\filldraw[black] (14,2) circle (2pt);
\end{tikzpicture}
\end{center}
\caption{Labeled line of $n$ hexagons}\label{hexagon_line}
\end{figure}
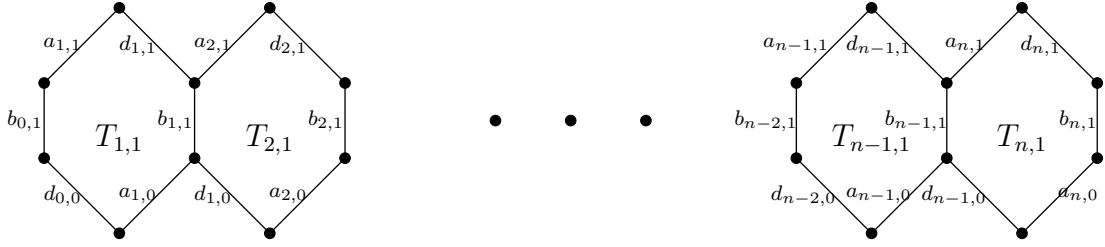


For each vertical segment $b_{i,1}$ ($0\le i\le n$), there is a unique perfect matching for $H_n$ containing $b_{i,1}$.  This is the matching

\noindent
$\{a_{1,1},d_{0,0},a_{2,1},d_{1,0},\ldots, a_{i,1},d_{i-1,0}, b_{i,1}, d_{i+1,1},a_{i+1,0}, \ldots
  d_{n-1,1},a_{n-1,0},d_{n, 1},a_{n,0}\}$ (omitting the $a_{1,1},d_{0,0}$ for ${i=0}$, and omitting the
$d_{n,1},a_{n,0}$ for $i=n$).
Let $A_i$ be the corresponding simplex (facet) in the perfect matching complex
$\match_p(H_n)$, $i \in \{0,1, \ldots, n\}$.
Let ${\cal N}(\match_p(H_n))$ be the nerve complex of
$\match_p(H_n)$ with vertices $A_i$.
For all $i$, $0\le i\le n-1$, $A_i$ contains the vertices $d_{n,1},a_{n,0}$, so
$\{A_0,A_1,A_2,\ldots, A_{n-1}\}$ is a simplex in ${\cal N}(\match_p(H_n))$.
For all $i$, $1\le i\le n$, $A_i$ contains the vertices $a_{1,1},d_{0,0}$, so
$\{A_1,A_2,\ldots, A_n\}$ is a simplex in ${\cal N}(\match_p(H_n))$.
Since $A_0\cap A_n=\emptyset$ in $\match_p(H_n)$, ${\cal N}(\match_p(H_n))$ is the suspension over the
simplex $\{A_1,A_2,\ldots, A_{n-1}\}$, and so is contractible. The conditions of the nerve lemma are satisfied because a nonempty intersection of simplices in a complex is contractible. Therefore, the perfect matching complex $\match_p(H_n)$ is contractible.
\end{proof}

\section{The $1 \times m \times n$ honeycomb graph}

\subsection{A short summary of discrete Morse theory}
Our subsequent calculations of homotopy type rely on discrete Morse theory. 
Developed by Forman, discrete Morse theory is a way to find the homotopy type of complexes by pairing faces of the complex~\cite{forman2002user}. These pairings correspond with a sequence of collapses on the complex, resulting in a homotopy equivalent cell complex. 

In what follows we will say two faces are \textit{paired} in place of the usual phrasing of two faces are \textit{matched} in a discrete Morse matching, in order to avoid unnecessary confusion with perfect matchings of a graph. 

\begin{definition} 
A \textbf{partial pairing} in a poset $P$ is a partial pairing on the underlying graph of the Hasse diagram of $P$. In other words, it is a subset $M \subseteq P \times P$ such that: 
\begin{itemize} 
\item $(a,b) \in M$ implies $a \prec b$ and 
\item each $c \in P$ belongs to at most one pair in $M$. 
\end{itemize} 
When $(a,b) \in M$ we write $b = u(a)$. A partial pairing is \textbf{acyclic} if there does not exist a cycle 
\[
a_1 \prec u(a_1) \succ a_2 \prec u(a_2) \succ \cdots \prec u(a_m) \succ a_1
\]
with $m \geq 2$ and $a_i \in P$ distinct. 
\end{definition} 
Given an acyclic partial pairing $M$ on a poset $P$, we call an element \textbf{critical} if it is unpaired. The main theorem of discrete Morse theory describes the essence of these sequences of collapses.

\begin{theorem}[\cite{forman2002user}]\label{DMT}
Let $\Delta$ be a polyhedral cell complex and let $M$ be an acyclic pairing on the face poset of $\Delta$. Let $c_i$ denote the number of critical $i$-dimensional cells of $\Delta$. The space $\Delta$ is homotopy equivalent to a cell complex $\Delta_c$ with $c_i$ cells of dimension $i$ for each $i \geq 0$, plus a single $0$-dimensional cell in the case where the empty set is paired in the matching. 
\end{theorem}

 A very simple way of constructing a pairing on a face poset is to choose a vertex and then pair each face that contains that vertex with the subface obtained by deleting that vertex. 

\begin{definition}\label{element_matching}\cite{higher_ind, jonsson-book}
Let $x$ be an arbitrary vertex of a simplicial complex $K$. The 
\textbf{ element pairing on $K$ using vertex $x$} is defined as: 
$$M(x)= \{ (\sigma,  \sigma \cup \{x\}) \mid x \notin \sigma,\  \sigma \cup \{x\} \in K \}.$$
\end{definition}


Throughout the paper we will use the property that a union of a sequence of element pairings is an acyclic pairing, as the following theorem claims.

\begin{theorem} (\cite[Proposition 2.10]{higher_ind}, \cite[Lemma 4.1]{jonsson-book}) \label{sequence_element_matchings}
Let $K$ be a simplicial complex and $\{x_1,x_2, \ldots, x_k\}$ be a subset of the vertex set of $K$. Let $K_0= K$, and for all $i \in \{1,2, \ldots, k\}$ define inductively:
\begin{align*}
    & M(x_i) = \{ (\sigma,  \sigma \cup \{x_i\}) \mid x_i \notin \sigma,\text{ and }  \sigma, \sigma \cup \{x_i\} \in K_{i-1} \}, \\
    & N(x_i) = \{ \sigma \in K_{i-1} \mid \sigma \in \eta \text{ for some } \eta \in M(x_i)\}, \text{ and}\\
    & K_i = K_{i-1} \setminus N(x_i).
    \end{align*}
Then  $\displaystyle \bigsqcup_{i=1}^{k} M(x_i)$ is an acyclic pairing on $K$.
\end{theorem}

\subsection{Perfect matchings in $H_{1 \times m \times n}$}
We label the hexagons and the edges of 
the $1\times m\times n$ honeycomb graph as follows.
Let $T_{i,j}$ denote the hexagon located in column $i$ ($1\le i\le m$) and row $j$ ($1\le j\le n$) starting from the bottom left hexagon; see Figures~\ref{honeycomb} and~\ref{edge-labels}. 
That is, for fixed $i$, the hexagons $T_{i,j}$, $1\le j\le n$, form a vertical
sequence with $T_{i,1}$ on the bottom and $T_{i,n}$ on the top.
For fixed $j$, the hexagons $T_{i,j}$, $1\le i\le m$, form a horizontal
sequence with $T_{1,j}$ on the left and $T_{m,j}$ on the right
(see Figure~\ref{honeycomb}).
The edges are labeled $a_{i,j}$, $b_{i,j}$, $d_{i,j}$, as in 
Figure~\ref{edge-labels}.
Here $a_{i,j}$ exists for $1\le i\le m$ and $0\le j\le n$;
$b_{i,j}$ exists for $0\le i\le m$ and $1\le j\le n$; and
$d_{i,j}$ exists for $0\le i\le m$ and $0\le j\le n$, except $d_{0,n}$ 
and $d_{m,0}$.

\begin{figure}[h]  
\begin{center}
\begin{tikzpicture}[scale=0.50]
\draw (0,4.5)--(1,6)--(0,7.5)--(1,9)--(3,9)--(4,7.5)--(6,7.5)--(7,6)--(9,6)--(10,4.5)--(9,3)--(10,1.5)--(9,0)--(7,0)--(6,1.5)--(4,1.5)--(3,3)--(1,3)--(0,4.5);
\draw (3,3)--(4,4.5)--(3,6)--(4,7.5);
\draw (6,1.5)--(7,3)--(6,4.5)--(7,6);
\draw (1,6)--(3,6);
\draw (4,4.5)--(6,4.5);
\draw (7,3)--(9,3);
\draw (2,7) node[above] {$T_{1,2}$};
\draw (5,5.5) node[above] {$T_{2,2}$};
\draw (8,4) node[above] {$T_{3,2}$};
\draw (2,4) node[above] {$T_{1,1}$};
\draw (5,2.5) node[above] {$T_{2,1}$};
\draw (8,1) node[above] {$T_{3,1}$};
\end{tikzpicture}
\end{center}
\caption{Honeycomb Labels \label{honeycomb}}
\end{figure}
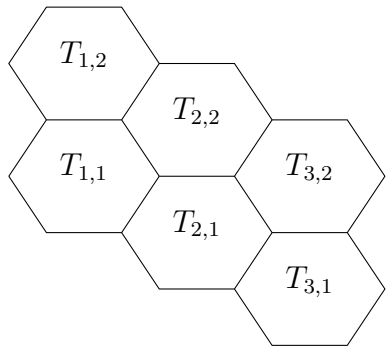

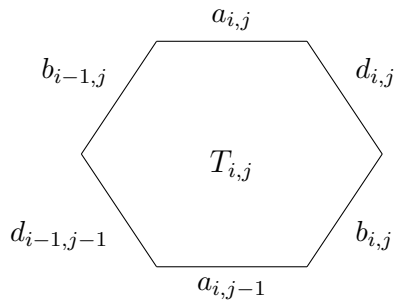
\begin{figure}[h]
\begin{center}
\begin{tikzpicture}
\draw (0,1.5)--(1,3) node[draw=none,fill=none,font=\small,midway,above left] 
       {$b_{i-1,j}$};
\draw (1,3)--(3,3) node[draw=none,fill=none,font=\small,midway,above] 
       {$a_{i,j}$};
\draw (3,3)--(4,1.5) node[draw=none,fill=none,font=\small,midway,above right] 
       {$d_{i,j}$};
\draw (3,0)--(4,1.5) node[draw=none,fill=none,font=\small,midway,below right] 
       {$b_{i,j}$};
\draw (3,0)--(1,0) node[draw=none,fill=none,font=\small,midway,below] 
       {$a_{i,j-1}$};
\draw (1,0)--(0,1.5) node[draw=none,fill=none,font=\small,midway,below left] 
       {$d_{i-1,j-1}$};
\draw (2,1) node[above] {$T_{i,j}$};
\end{tikzpicture}
\end{center}
\caption{Edge Labels \label{edge-labels}}
\end{figure}

\begin{remark} 
Recall
each perfect matching $P$  on $H_{1 \times m \times n}$  is identified with a 
plane partition, represented by a $1\times m$ matrix.  
Thus, we will denote $P$ as
$P = (h_1, h_2, \dots, h_m)$ where $n \geq h_1 \geq \cdots \geq h_m \geq 0$ and $h_1, \ldots, h_m \in \mathbb{Z}_{\geq 0}$. Notice that we can think of $h_1, \dots, h_m$ as heights of the respective columns of cubes. 
See Figure~\ref{fig:honeycomb_partition} for an example.

The matching $(h_1,h_2,\ldots, h_m)$, contains the following edges:  
\begin{itemize}
    \item $a_{i,j}$, if the edge is visible in the top of a cube or if column $i$ has no cubes and $a_{i,j}$ is the bottom horizontal edge in the column;
    \item $b_{i,j}$, if the edge is visible on the front right side of a cube or $b_{0,j}$ if the hexagon in column 1 is above the top cube in column 1;
    \item $d_{i,j}$, if the edge is in the front left side of a cube for any column.
\end{itemize}
\end{remark}
In terms of the plane partition, we get the following description of edges
in the perfect matching.
\begin{proposition}\label{remark:edge_direction}
Let $P$ be the perfect matching corresponding to the plane partition
$(h_1,h_2,\ldots, h_m)$.  Then
\begin{itemize}
\item $a_{i,j}\in P$ if and only if $j=h_i$;
\item $b_{i,j}\in P$ if and only if  
     ($h_i\ge j>h_{i+1}$) or ($i=0$ and $j>h_1$) or ($i=m$ and $j \leq h_m)$;
\item $d_{i,j}\in P$ if and only if 
      ($j> h_i$) or ($j<h_{i+1}$).
\end{itemize}
\end{proposition}

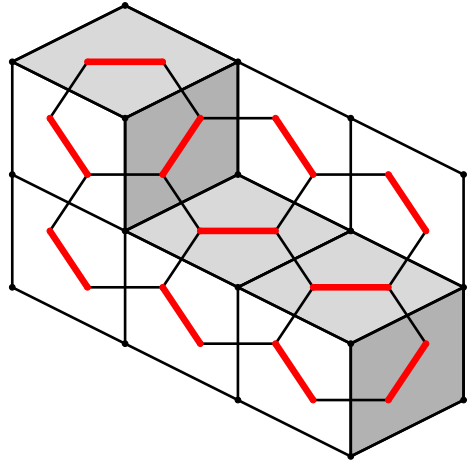
\begin{figure} 
\begin{center}
\begin{tikzpicture}[scale = 0.5]

\draw[line width = 1 pt, fill = gray!30] (1,9) -- (4,7.5) -- (7, 9) -- (4, 10.5) --(1,9);

\draw[line width = 1 pt, fill = gray!60] (4, 7.5) -- (4,4.5) -- (7,6) -- (7, 9) --(4,7.5);

\draw[line width = 1 pt, fill = gray!30] (4,4.5) -- (7,6) -- (10, 4.5) -- (7, 3) --(4, 4.5);

\draw[line width = 1 pt, fill = gray!30] (10,4.5) -- (7,3) -- (10, 1.5) -- (13, 3) --(10, 4.5);

\draw[line width = 1 pt, fill = gray!60] (13,3) -- (10,1.5) -- (10,-1.5) -- (13,0) --(13,3);

\draw[line width = 1 pt] (2,7.5) -- (3,9);  
\draw[line width = 1 pt] (5,9) -- (6,7.5); 
\draw[line width = 1 pt] (6,7.5) -- (8,7.5); 
\draw[line width = 1 pt] (9,6) -- (11,6); 
\draw[line width = 1 pt] (3,6) -- (5,6); 
\draw[line width = 1 pt] (5,6) -- (6,4.5); 
\draw[line width = 1 pt] (8,4.5) -- (9,6); 
\draw[line width = 1 pt] (8,4.5) -- (9,3); 
\draw[line width = 1 pt] (11,3) -- (12,4.5); 
\draw[line width = 1 pt] (11,3) -- (12,1.5); 
\draw[line width = 1 pt] (9,3) -- (8,1.5); 
\draw[line width = 1 pt] (8,1.5) -- (6,1.5);
\draw[line width = 1 pt] (5,3) -- (6,4.5);
\draw[line width = 1 pt] (3,3) -- (5,3);
\draw[line width = 1 pt] (2,4.5) -- (3,6);
\draw[line width = 1 pt] (9,0) -- (11,0);

\draw[red, line width = 2.5 pt] (3,9) -- (5,9); 
\draw[red, line width = 2.5 pt] (3,6) -- (2,7.5); 
\draw[red, line width = 2.5 pt] (3,3) -- (2, 4.5); 
\draw[red, line width = 2.5 pt] (5,6) -- (6, 7.5); 
\draw[red, line width = 2.5 pt] (9,6) -- (8, 7.5); 
\draw[red, line width = 2.5 pt] (6,4.5) -- (8, 4.5); 
\draw[red, line width = 2.5 pt] (6, 1.5) -- ( 5,3); 
\draw[red, line width = 2.5 pt] (11,6) -- (12, 4.5); 
\draw[red, line width = 2.5 pt] (9,3) -- (11,3); 
\draw[red, line width = 2.5 pt] (9,0) -- (8,1.5); 
\draw[red, line width = 2.5 pt] (11,0) -- (12,1.5);


\filldraw[red] (2,7.5) circle (0.075cm); 
\filldraw[red] (2,4.5) circle (0.075cm); 
\filldraw[red] (3,3) circle (0.075cm); 
\filldraw[red] (3,6) circle (0.075cm); 
\filldraw[red] (3,9) circle (0.075cm);
\filldraw[red] (5,9) circle (0.075cm); 
\filldraw[red] (5,6) circle (0.075cm); 
\filldraw[red] (5,3) circle (0.075cm); 
\filldraw[red] (6,1.5) circle (0.075cm); 
\filldraw[red] (6,4.5) circle (0.075cm); 
\filldraw[red] (6,7.5) circle (0.075cm); 
\filldraw[red] (8,1.5) circle (0.075cm); 
\filldraw[red] (8,4.5) circle (0.075cm); 
\filldraw[red] (8,7.5) circle (0.075cm); 
\filldraw[red] (9,6) circle (0.075cm); 
\filldraw[red] (9,3) circle (0.075cm); 
\filldraw[red] (9,0) circle (0.075cm); 
\filldraw[red] (11,0) circle (0.075cm); 
\filldraw[red] (11,3) circle (0.075cm); 
\filldraw[red] (11,6) circle (0.075cm); 
\filldraw[red] (12,4.5) circle (0.075cm); 
\filldraw[red] (12,1.5) circle (0.075cm); 


\draw[line width = 1 pt] (1,3) -- (1,9);
\draw[line width = 1 pt] (4,1.5) -- (4,7.5);  
\draw[line width = 1 pt] (7,0) -- (7,3);
\draw[line width = 1 pt] (7,6) -- (7,9);
\draw[line width = 1 pt] (10,-1.5) -- (10,1.5);
\draw[line width = 1 pt] (10,4.5) -- (10,7.5);
\draw[line width = 1 pt] (13,0) -- (13,6);  
\draw[line width = 1 pt] (1,9) -- (4,10.5);
\draw[line width = 1 pt] (4,10.5) -- (13,6);
\draw[line width = 1 pt] (13,3) -- (7,6);
\draw[line width = 1 pt] (4,7.5) -- (1,9);
\draw[line width = 1 pt] (1,6) -- (10,1.5);  
\draw[line width = 1 pt] (4,7.5) -- (7,9);
\draw[line width = 1 pt] (4,4.5) -- (7,6); 
\draw[line width = 1 pt] (1,3) -- (10,-1.5);
\draw[line width = 1 pt] (10,-1.5) -- (13,0); 
\draw[line width = 1 pt] (10,1.5) -- (13,3); 
\draw[line width = 1 pt] (7,3) -- (10,4.5);  

\filldraw (1,3) circle (0.075cm);
\filldraw (1,6) circle (0.075cm);
\filldraw (1,9) circle (0.075cm);
\filldraw (4,10.5) circle (0.075cm);
\filldraw (4,7.5) circle (0.075cm);
\filldraw (4,4.5) circle (0.075cm);
\filldraw (4,1.5) circle (0.075cm);
\filldraw (7,0) circle (0.075cm);
\filldraw (7,3) circle (0.075cm);
\filldraw (7,6) circle (0.075cm);
\filldraw (7,9) circle (0.075cm);
\filldraw (10,-1.5) circle (0.075cm);
\filldraw (10,1.5) circle (0.075cm);
\filldraw (10,4.5) circle (0.075cm);
\filldraw (10,7.5) circle (0.075cm);
\filldraw (13,0) circle (0.075cm);
\filldraw (13,3) circle (0.075cm);
\filldraw (13,6) circle (0.075cm);

\end{tikzpicture}
\end{center}
\caption{\label{figure4}Pictured above is $H_{1 \times 3 \times 2}$ along with a perfect matching highlighted in thicker red edges and overlaid with the corresponding plane partition $(2,1,1)$.}
\label{fig:honeycomb_partition}
\end{figure}

\begin{definition}\label{significant_edges}
For a $1 \times m \times n$ honeycomb graph, the edges $d_{i,j}$ at the intersection of $T_{i,j}$ and $T_{i+1,j+1}$ for $1 \leq i \leq m-1$ and $1 \leq j \leq n-1$ are called \textbf{significant edges}. See Figure~\ref{fig_significant_edges}.
\end{definition}

\begin{figure}[h] 
\begin{center}
\begin{tikzpicture}[scale = 0.4] 

\draw[line width = 1 pt] (2,7.5) -- (3,9);  
\draw[red, line width = 2.5 pt] (5,9) -- (6,7.5); 
\draw[line width = 1 pt] (6,7.5) -- (8,7.5); 
\draw[line width = 1 pt] (9,6) -- (11,6); 
\draw[line width = 1 pt] (3,6) -- (5,6); 
\draw[red, line width = 2.5 pt] (5,6) -- (6,4.5); 
\draw[line width = 1 pt] (8,4.5) -- (9,6); 
\draw[red, line width = 2.5 pt] (8,4.5) -- (9,3); 
\draw[line width = 1 pt] (11,3) -- (12,4.5); 
\draw[red, line width = 2.5 pt] (11,3) -- (12,1.5); 
\draw[line width = 1 pt] (9,3) -- (8,1.5); 
\draw[line width = 1 pt] (8,1.5) -- (6,1.5);
\draw[line width = 1 pt] (5,3) -- (6,4.5);
\draw[line width = 1 pt] (3,3) -- (5,3);
\draw[line width = 1 pt] (2,4.5) -- (3,6);
\draw[line width = 1 pt] (9,0) -- (11,0);

\draw[line width = 1 pt] (3,9) -- (5,9); 
\draw[line width = 1 pt] (3,6) -- (2,7.5); 
\draw[red, dashed, line width = 2.5 pt] (3,3) -- (2, 4.5); 
\draw[line width = 1 pt] (5,6) -- (6, 7.5); 
\draw[red, line width = 2.5 pt] (9,6) -- (8, 7.5); 
\draw[line width = 1 pt] (6,4.5) -- (8, 4.5); 
\draw[line width = 1 pt] (6, 1.5) -- ( 5,3); 
\draw[red, line width = 2.5 pt] (11,6) -- (12, 4.5); 
\draw[line width = 1 pt] (9,3) -- (11,3); 
\draw[line width = 1 pt] (9,0) -- (8,1.5); 
\draw[line width = 1 pt] (11,0) -- (12,1.5);


\draw[line width = 1 pt] (2,14.5) -- (3,16);  
\draw[line width = 1 pt] (5,16) -- (6,14.5); 
\draw[line width = 1 pt] (6,14.5) -- (8,14.5); 
\draw[line width = 1 pt] (9,13) -- (11,13); 
\draw[line width = 1 pt] (3,13) -- (5,13); 
\draw[red, line width = 2.5 pt] (5,13) -- (6,11.5); 
\draw[line width = 1 pt] (8,11.5) -- (9,13); 
\draw[red, line width = 2.5 pt] (8,11.5) -- (9,10); 
\draw[line width = 1 pt] (11,10) -- (12,11.5); 
\draw[red, line width = 2.5 pt] (11,10) -- (12,8.5); 
\draw[line width = 1 pt] (9,10) -- (8,8.5); 
\draw[line width = 1 pt] (8,8.5) -- (6,8.5);
\draw[line width = 1 pt] (5,10) -- (6,11.5);
\draw[line width = 1 pt] (3,10) -- (5,10);
\draw[line width = 1 pt] (2,11.5) -- (3,13);
\draw[line width = 1 pt] (9,7) -- (11,7);

\draw[line width = 1 pt] (3,16) -- (5,16); 
\draw[line width = 1 pt] (3,13) -- (2,14.5); 
\draw[line width = 1 pt] (3,10) -- (2, 11.5); 
\draw[line width = 1 pt] (5,13) -- (6, 14.5); 
\draw[line width = 1 pt] (9,13) -- (8, 14.5); 
\draw[line width = 1 pt] (6,11.5) -- (8, 11.5); 
\draw[red, line width = 2.5 pt] (6, 8.5) -- ( 5,10); 
\draw[line width = 1 pt] (11,13) -- (12, 11.5); 
\draw[line width = 1 pt] (9,10) -- (11,10); 
\draw[red, line width = 2.5 pt] (9,7) -- (8,8.5); 
\draw[line width = 1 pt] (11,7) -- (12,8.5);


\draw[line width = 1 pt] (12,2.5) -- (13,4);  
\draw[red, line width = 2.5 pt](15,4) -- (16,2.5); 
\draw[line width = 1 pt] (16,2.5) -- (18,2.5); 
\draw[line width = 1 pt] (19,1) -- (21,1); 
\draw[line width = 1 pt] (13,1) -- (15,1); 
\draw[red, line width = 2.5 pt] (15,1) -- (16,-0.5); 
\draw[line width = 1 pt] (18,-0.5) -- (19,1); 
\draw[red, line width = 2.5 pt] (18,-0.5) -- (19,-2); 
\draw[line width = 1 pt] (21,-2) -- (22,-0.5); 
\draw[line width = 1 pt] (21,-2) -- (22,-3.5); 
\draw[line width = 1 pt] (19,-2) -- (18,-3.5); 
\draw[line width = 1 pt] (18,-3.5) -- (16,-3.5);
\draw[line width = 1 pt] (15,-2) -- (16,-0.5);
\draw[line width = 1 pt] (13,-2) -- (15,-2);
\draw[line width = 1 pt] (12,-0.5) -- (13,1);
\draw[line width = 1 pt] (19,-5) -- (21,-5);

\draw[line width = 1 pt] (13,4) -- (15,4); 
\draw[red, line width = 2.5 pt] (13,1) -- (12,2.5); 
\draw[line width = 1 pt] (13,-2) -- (12, -0.5); 
\draw[line width = 1 pt] (15,1) -- (16, 2.5); 
\draw[red, line width = 2.5 pt] (19,1) -- (18, 2.5); 
\draw[line width = 1 pt] (16,-0.5) -- (18, -0.5); 
\draw[line width = 1 pt] (16, -3.5) -- ( 15,-2); 
\draw[line width = 1 pt] (21,1) -- (22, -0.5); 
\draw[line width = 1 pt] (19,-2) -- (21,-2); 
\draw[line width = 1 pt] (19,-5) -- (18,-3.5); 
\draw[line width = 1 pt] (21,-5) -- (22,-3.5);


\draw[line width = 1 pt] (12,9.5) -- (13,11);  
\draw[line width = 1 pt] (15,11) -- (16,9.5); 
\draw[line width = 1 pt] (16,9.5) -- (18,9.5); 
\draw[line width = 1 pt] (19,8) -- (21,8); 
\draw[line width = 1 pt] (13,8) -- (15,8); 
\draw[red, line width = 2.5 pt] (15,8) -- (16,6.5); 
\draw[line width = 1 pt] (18,6.5) -- (19,8); 
\draw[red, line width = 2.5 pt] (18,6.5) -- (19,5); 
\draw[line width = 1 pt] (21,5) -- (22,6.5); 
\draw[line width = 1 pt] (21,5) -- (22,3.5); 
\draw[line width = 1 pt] (19,5) -- (18,3.5); 
\draw[line width = 1 pt] (18,3.5) -- (16,3.5);
\draw[line width = 1 pt] (15,5) -- (16,6.5);
\draw[line width = 1 pt] (13,5) -- (15,5);
\draw[line width = 1 pt] (12,6.5) -- (13,8);
\draw[line width = 1 pt] (19,2) -- (21,2);

\draw[line width = 1 pt] (13,11) -- (15,11); 
\draw[red, line width = 2.5 pt] (13,8) -- (12,9.5); 
\draw[red, line width = 2.5 pt] (13,5) -- (12, 6.5); 
\draw[line width = 1 pt] (15,8) -- (16, 9.5); 
\draw[line width = 1 pt] (19,8) -- (18, 9.5); 
\draw[line width = 1 pt] (16,6.5) -- (18, 6.5); 
\draw[red, line width = 2.5 pt] (16, 3.5) -- ( 15,5); 
\draw[red, dashed, line width = 2.5 pt](21,8) -- (22, 6.5); 
\draw[line width = 1 pt] (19,5) -- (21,5); 
\draw[red, line width = 2.5 pt] (19,2) -- (18,3.5); 
\draw[line width = 1 pt] (21,2) -- (22,3.5);

\draw [decorate, line width = 2 pt,
    decoration = {brace}] (1,3) --  (1,17);
 
\draw [decorate, line width = 2 pt,
    decoration = {calligraphic brace}] (2,17) --  (24,17);
 
\node at (0,10) {$n$}; 
\node at (13,18) {$m$}; 
\node at (22,7.5) {$y$}; 
\node at (2.25,3.5) {$x$}; 
\node at (19,-1) {$z$};

\end{tikzpicture} 
\caption{The $1 \times m \times n$ honeycomb graph with significant edges 
highlighted in thicker red.}\label{fig_significant_edges}
 \end{center}
\end{figure}


Directly from Definition~\ref{perfect_matching_complex} we derive the following  properties of facets of $\match_p(H_{1 \times m \times n}).$

\begin{lemma}\label{lemma_xy}
Consider the honeycomb graph $H=H_{1 \times m \times n}$, $m, n \in \mathbb{N}$, $m,n \ge 2$, and its perfect matching complex $\match_p(H).$ 
Consider vertices of $\match_p(H)$ that  correspond to edges $x=d_{0,0}$ and 
$y=d_{m,n}$ (Figure~\ref{fig_significant_edges}). Let  $P=(h_1, \ldots, h_m)$ be an arbitrary perfect matching on $H$. Then vertex $x$ belongs to $P$ if and only if $P \neq (\underbrace{0, \ldots, 0}_{m}),$ while vertex  $y$ belongs to $P$ if and only if $P \neq (\underbrace{n, \ldots, n}_{m}).$ 
\end{lemma}

\begin{proof} 
By Proposition~\ref{remark:edge_direction},
the vertex $y=d_{m,n}$ belongs to $P$ if and only if $h_m <  n$, which occurs if only if $P \neq (\underbrace{n, \ldots, n}_{m}).$ 
The vertex $x=d_{0,0}$ belongs to $P$ if and only if $h_1 >0$, which occurs 
if and only if $P \neq (\underbrace{0, \ldots, 0}_{m}).$ 
\end{proof}

\begin{lemma}\label{intersection_all0_alln}
Consider the honeycomb graph $H=H_{1 \times m \times n}$, $m, n \in \mathbb{N}$, $m,n \ge 2$, and its perfect matching complex $\match_p(H).$ Then the  intersection of facet $(\underbrace{0, \ldots, 0}_{m})$ and facet $(\underbrace{n, \ldots, n}_{m})$ is a simplex whose set of vertices corresponds to the set of significant edges. 
\end{lemma}

\begin{proof}
If $P$ equals either $(0,\ldots, 0)$ or $(n,\ldots, n)$, then, according to Proposition~\ref{remark:edge_direction}, for $1\le i\le m-1$ and $1\le j\le n-1$, $P$ contains neither $a_{i,j}$ nor $b_{i,j+1}$.  Each interior vertex of the honeycomb graph lies on exactly three edges, $a_{i,j}$, $b_{i,j+1}$ and $d_{i,j}$ ($1\le i\le m-1$, $1\le j\le n-1$).  Thus, to cover this interior vertex, $P$ must contain the significant edge $d_{i,j}$. 

We check that the two perfect matchings share no other edges.  Clearly, they share no edge $a_{i,j}$.  The matching $(0,\ldots, 0)$ contains $b_{0,j}$ ($1\le j\le n$) but no other $b_{i,j}$.  The matching $(n,\ldots, n)$ contains $b_{m,j}$ ($1\le j\le n$) but no other $b_{i,j}$.  So the two share no edge $b_{i,j}$. Finally, $(n,\ldots, n)$ contains no $d_{i,n}$ and no $d_{m,j}$, while  $(0,\ldots, 0)$ contains no $d_{0,j}$ and no $d_{i,0}$. So the intersection of the perfect matchings $(0,\ldots, 0)$ and $(n,\ldots, n)$ is exactly the simplex with vertices corresponding to the significant edges.
\end{proof}

We will use this Lemma in the next two sections as we create Morse pairings.

\subsection{The $1 \times 2 \times n$ honeycomb graph}
Before proving the main result for $1\times m\times n$ honeycomb graphs, 
we consider the homotopy type in the special case $H_{1 \times 2 \times n}$. 


\begin{figure}[h]
\begin{center}
\begin{tikzpicture}[scale = 0.5]

\draw[line width = 1 pt] (2,7.5) -- (3,9);  
\draw[red, line width = 2.5 pt] (5,9) -- (6,7.5); 
\draw[line width = 1 pt] (6,7.5) -- (8,7.5); 
\draw[line width = 1 pt] (3,6) -- (5,6); 
\draw[red, line width = 2.5 pt] (5,6) -- (6,4.5); 
\draw[line width = 1 pt] (8,4.5) -- (9,6); 
\draw[line width = 1 pt] (8,4.5) -- (9,3); 
\draw[line width = 1 pt] (9,3) -- (8,1.5); 
\draw[line width = 1 pt] (8,1.5) -- (6,1.5);
\draw[line width = 1 pt] (5,3) -- (6,4.5);
\draw[line width = 1 pt] (3,3) -- (5,3);
\draw[line width = 1 pt] (2,4.5) -- (3,6);

\node at (5,8) {$d_{1,2}$}; 
\node at (5,5) {$d_{1,1}$}; 

\draw[line width = 1 pt] (3,9) -- (5,9); 
\draw[line width = 1 pt] (3,6) -- (2,7.5); 
\draw[red, dashed, line width = 2.5 pt] (3,3) -- (2, 4.5); 
\draw[line width = 1 pt] (5,6) -- (6, 7.5); 
\draw[line width = 1 pt] (9,6) -- (8, 7.5); 
\draw[line width = 1 pt] (6,4.5) -- (8, 4.5); 
\draw[line width = 1 pt] (6, 1.5) -- ( 5,3); 


\draw[line width = 1 pt] (2,14.5) -- (3,16);  
\draw[line width = 1 pt] (5,16) -- (6,14.5); 
\draw[line width = 1 pt] (6,14.5) -- (8,14.5); 
\draw[line width = 1 pt] (3,13) -- (5,13); 
\draw[red, line width = 2.5 pt] (5,13) -- (6,11.5); 
\draw[line width = 1 pt] (8,11.5) -- (9,13); 
\draw[line width = 1 pt] (8,11.5) -- (9,10); 
\draw[line width = 1 pt] (9,10) -- (8,8.5); 
\draw[line width = 1 pt] (8,8.5) -- (6,8.5);
\draw[line width = 1 pt] (5,10) -- (6,11.5);
\draw[line width = 1 pt] (3,10) -- (5,10);
\draw[line width = 1 pt] (2,11.5) -- (3,13);

\node at (6.2,12.5) {$d_{1,n}$}; 

\draw[line width = 1 pt] (3,16) -- (5,16); 
\draw[line width = 1 pt] (3,13) -- (2,14.5); 
\draw[line width = 1 pt] (3,10) -- (2, 11.5); 
\draw[line width = 1 pt] (5,13) -- (6, 14.5); 
\draw[red, dashed, line width = 2.5 pt] (9,13) -- (8, 14.5); 
\draw[line width = 1 pt] (6,11.5) -- (8, 11.5); 
\draw[red, line width = 2.5 pt] (6, 8.5) -- ( 5,10); 

\node at (6.45,9.5) {$d_{1,n-1}$};

 
 
\node at (9,13.75) {$y$}; 
\node at (2.25,3.5) {$x$}; 

\end{tikzpicture}
 \caption{Case $1 \times 2 \times n$.}\label{fig_1x2xn}
\end{center}
\end{figure}
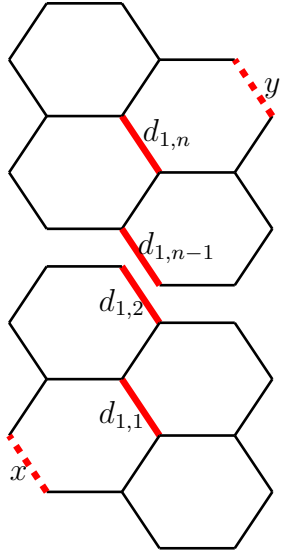


\begin{theorem}\label{1x2xn_sphere}
Let $H_{1 \times 2 \times n}$ be the honeycomb graph with $n \ge 2$.  Then
\[
\match_p(H_{1 \times 2 \times n}) \simeq S^{n-1}. 
\]
\end{theorem}
\begin{proof} Let $H=H_{1 \times 2 \times n}$.   We construct a Morse pairing on the face poset of $\match_p(H)$ by using two element pairings:
\begin{itemize}
\item[(1)] construct element pairing $M(x)$ using vertex $x$ (see Figure~\ref{fig_1x2xn}); then
\item[(2)]  on the set of unpaired faces, construct element pairing $M(y)$ using vertex $y$.
\end{itemize}

From Theorem~\ref{sequence_element_matchings} we know that the union of a sequence of element pairings is an acyclic pairing. Therefore, $M(x) \cup M(y)$ is an acyclic pairing on the face poset of $\match_p(H)$. We will prove that there is only one face of $\match_p(H)$  that is unpaired after $M(x) \cup M(y)$.

From Lemma~\ref{lemma_xy} we know that vertex $x$ belongs to a perfect matching $P=(h_1, h_2)$ if and only if $P \neq (0, 0),$ while vertex  $y$ belongs to $P$ if and only if $P \neq (n, n).$ Therefore, complex $\match_p(H)$ can be seen as the union of two cones with apices $x$ and $y$. From this representation we observe that a face $\tau \in \match_p(H)$ is unpaired after $M(x) \cup M(y)$ if and only if $\tau$ has the following structure:
$$\tau = \{y\} \cup \sigma,$$
where the face $\sigma \in \match_p(H)$ satisfies:
\begin{equation}\label{sigmas}
x, y \notin \sigma, \ \{y\} \cup \sigma \in  \match_p(H), \ \{x\} \cup \sigma \in  \match_p(H), \text{ and }  \{x,y\} \cup \sigma \notin  \match_p(H).
\end{equation}
Let $\sigma \in \match_p(H)$ be an arbitrary face that satisfies (\ref{sigmas}). If there exists a perfect matching $P \notin \{(0,0), (n,n)\}$ such that $\sigma \subseteq P,$ 
then $\sigma \cup \{x,y\}$ is also contained in $P$,  which contradicts the last condition in (\ref{sigmas}). The only perfect matchings that might contain $\sigma$ are $(0,0)$ and $(n,n)$. Further, since $x \notin (0,0)$ and $\{x\} \cup \sigma \in  \match_p(H)$, we conclude that $\sigma \subseteq (n,n).$ Similarly, conditions $y \notin (n,n)$ and $\{y\} \cup \sigma \in  \match_p(H)$ imply that $\sigma \subseteq (0,0).$ Therefore,
$$\sigma \subseteq (0,0) \cap (n,n) =  \{d_{1,1}, \ldots, d_{1,n-1}\},$$
where the second equality follows from  Lemma~\ref{intersection_all0_alln}.


Suppose that $\rho \subsetneq \{d_{1,1}, \ldots, d_{1,n-1}\}$ is a proper subset. Then, there exists some index $i \in \{1, \ldots, n-1\}$ such that $d_{1,i} \notin \rho$. The perfect matching $(i, i) \notin \{(0,0), (n,n)\}$ contains $\{d_{1,1}, d_{1,2}, \dots, d_{1,n-1}\} \smallsetminus \{d_{1,i}\}$. Therefore $(i,i)$ is a perfect matching that contains $\rho$.

Hence, there is only one critical cell, $\tau =\{y\} \cup \{d_{1,1}, \ldots, d_{1,n-1}\}.$ From Theorem~\ref{DMT}, complex $\match_p(H)$ is homotopy equivalent to a CW-complex with one  $(n-1)$-dimensional cell and one $0$-cell (because the empty set is paired with set $\{x\}$). Hence, $\match_p(H) \simeq S^{n-1}$.
\end{proof}

\subsection{The $1\times m\times n$ honeycomb graph, $m,n\ge 3$}

\begin{theorem}\label{1xmxn_contractible}
Let $H_{1 \times m \times n}$ be the honeycomb graph of dimension $1 \times m \times n$, with $m, n \in \mathbb{N}$, and $m, n \ge 3$.  Then the perfect matching complex 
$\match_p(H_{1 \times m \times n})$ is contractible.
\end{theorem}
\begin{proof}
Let $H=H_{1\times m\times n}$.
We construct a discrete Morse pairing for $\match_p(H)$. 
Begin with the two element pairings as in Theorem~\ref{1x2xn_sphere}, and
extend it with one significant edge: 
\begin{enumerate} 
\item[(1)] construct an element pairing $M(x)$ using vertex $x=d_{0,0}$; then
\item[(2)] on the set of unpaired faces, construct element pairing $M(y)$ using
          vertex $y=d_{m,n}$; then
\item[(3)] on the set of unpaired faces, construct element pairing 
          $M(d_{m-1,1})$ using vertex $z=d_{m-1,1}$. 
\end{enumerate} 
(See Figure~\ref{fig_significant_edges} for edge labels.)

From Theorem~\ref{sequence_element_matchings} we know that the union of a 
sequence of element pairings is an acyclic pairing; therefore 
$M(x) \cup M(y) \cup M(z)$ is an acylic pairing on the face poset of 
$\match_p(H)$. 

As in the proof of Theorem~\ref{1x2xn_sphere} we see that a face $\tau \in \match_p(H)$ is unpaired after $M(x) \cup M(y)$ if and only if $\tau$ has the following structure:
$$\tau = \{y\} \cup \sigma,$$
where $\sigma \subseteq (0,\ldots,0) \cap (n,\ldots,n)$ and $\tau\cup\{x\}\not\in
\match_p(H)$.

Let $N$ be the set of faces in $\match_p(H)$ that are unpaired after
$M(x)\cup M(y)$. Reasoning as in the proof of Theorem~\ref{1x2xn_sphere} we get that
\begin{eqnarray*}N =
\{\tau\in \match_p(H):  \tau=\{y\}\cup\sigma &\mbox{for some}& 
\sigma \subseteq (0,\ldots,0) \cap (n,\ldots,n) \\ && \mbox{and }
\tau\cup\{x\}\not\in\match_p(H)\}.
\end{eqnarray*}
Note that $\sigma \subseteq (0,\ldots,0) \cap (n,\ldots,n)$ implies that $x, y \notin \sigma$.
We claim that the final element pairing $M(z)$ pairs all faces in $N$, and therefore the homotopy type of the perfect matching complex is contractible.
Thus, we wish to show $z\not\in \tau$, $\tau\in N$ if and only if
$\tau\cup\{z\}\in N$.

($\Rightarrow$)  Assume $\tau=\{y\}\cup\sigma \in N$, $z\not\in \tau$.  
Then $\sigma \subseteq (0,\ldots,0) \cap (n,\ldots,n)$ and 
$\tau\cup\{x\}\not\in\match_p(H)$.  
Consider $\tau\cup\{z\}=\{y\} \cup (\sigma\cup\{z\})$.
Since $z\in (0,\ldots,0) \cap (n,\ldots,n)$, as $z$ is a significant
edge, $\sigma\cup \{z\}\subseteq (0,\ldots,0) \cap (n,\ldots,n)$.  
This also shows that $\tau\cup \{z\}\subseteq (0,\ldots,0)$, and hence is
in $\match_p(H)$.
Also, ${(\tau\cup \{z\})\cup\{x\}\not\in \match_p(H)}$, since its 
subset $\tau\cup\{x\}$ is  not in $\match_p(H)$.
So $\tau\cup\{z\}\in N$.

($\Leftarrow$)  Assume $\tau\cup\{z\}\in N$ ($z\not\in \tau$).
Then $\tau\cup\{z\}=\{y\}\cup (\sigma\cup\{z\})$ where
$\sigma\cup\{z\}\subseteq (0,\ldots,0) \cap (n,\ldots,n)$ and
$(\tau\cup\{z\})\cup\{x\}\not\in \match_p(H)$.
Clearly, $\tau=\{y\} \cup \sigma\in \match_p(H)$, with 
$\sigma\subseteq (0,\ldots,0) \cap (n,\ldots,n)$. 
We need to show that $\tau\cup\{x\}=\sigma\cup \{x,y\}\not\in \match_p(H)$.
We prove this by contradiction.

Assume $\sigma\cup \{x,y\}\in\match_p(H)$, but 
$(\sigma\cup \{z\})\cup \{x,y\}\not\in\match_p(H)$.  
Say $\sigma\cup \{x,y\}\subseteq (h_1,h_2,\ldots, h_m)$.
Since $(\sigma\cup \{z\})\cup \{x,y\}\not\subseteq (h_1,h_2,\ldots, h_m)$,
by Proposition~\ref{remark:edge_direction} (applied to $z=d_{m-1,1}$), 
$h_{m-1}\ge 1$ and $h_m\le 1$.

Case 1.  $h_{m-1}\ge 2$ and $h_m\le 1$.  We claim that in this case
$(\sigma\cup\{z\})\cup\{x,y\}\subseteq (h_1,h_2,\ldots, h_{m-1},2)$.
First note that $x$ and $y$ are in $(h_1,h_2,\ldots, h_{m-1},2)$, since
each is in every perfect matching except $(0,\ldots,0)$ (in the case of $x$)
and $(n,\ldots, n)$ (in the case of $y$).  
Since $\sigma\subseteq (0,\ldots, 0)\cap(n,\ldots, n)$, all other 
elements of 
$(\sigma\cup\{z\})\cup\{x,y\}$ are significant edges, that is,
edges of the form $d_{i,j}$, $1\le i\le m-1$, $1\le j\le n-1$.  
By Proposition~\ref{remark:edge_direction}, $d_{i,j}\in (h_1,h_2,\ldots, h_{m-1},2)$
if and only if $j>h_i$ or $j<h_{i+1}$.  (We are assuming $i>0$.)
In particular, $1<2$, so $z=d_{m-1,1}\in (h_1,\ldots, h_{m-1},2)$.
In addition, we know 
$\sigma\cup \{x,y\}\subseteq (h_1,h_2,\ldots, h_m)$, 
so for $d_{i,j}\in \sigma$ with $i<m-1$, $d_{i,j}$ satisfies the criterion
for $(h_1,h_2,\ldots, h_{m-1}, 2)$.
On the other hand, for
$d_{m-1,j}\in \sigma$, $j\ge 2>h_m$, so
$d_{m-1,j}\in(h_1,h_2,\ldots, h_m)$ implies $j>h_{m-1}$.
Thus, $d_{m-1,j}$ also satisfies the criterion
for $(h_1,h_2,\ldots, h_{m-1}, 2)$.  So in Case 1, we conclude that
$(\tau\cup\{z\})\cup \{x\} = (\sigma\cup \{z\})\cup \{x,y\}
\in\match_p(H)$, a contradiction.

Case 2. $h_{m-1}=1$ and $h_m\le 1$.  We claim that in this case
$(\sigma\cup\{z\})\cup\{x,y\}\subseteq (h_1,h_2,\ldots, h_{m-2},0,0)$.
As in Case 1, $x$ and $y$ are in $(h_1,h_2,\ldots, h_{m-2},0,0)$. Note that since $m\ge 3$ and $h_{m-1}=1$, we have $h_{m-2}\ge 1$. 
Also, $z=d_{m-1,1}$ is in $(h_1,h_2,\ldots, h_{m-2},0,0)$, since 
$1>0$.
Again, consider the significant edges $d_{i,j}$ in $\sigma$. We know these $d_{i,j}$ are in $(h_1,h_2,\ldots, h_{m-1},1,h_m)$, so for $i < m-2$ it is obvious that $d_{i,j} \in (h_1,h_2,\ldots, h_{m-2},0,0)$. For $i= m-2$, the condition from Proposition~\ref{remark:edge_direction} reduces to $j > h_{m-2}$, so $d_{m-2,j} \in (h_1,h_2,\ldots, h_{m-2},0,0)$ again. Finally, for $i= m-1$,  Proposition~\ref{remark:edge_direction} implies  $j > 1 >0,$ and $d_{i,j} \in (h_1,h_2,\ldots, h_{m-2},0,0)$.

Again we have shown that
$(\tau\cup\{z\})\cup \{x\} = (\sigma\cup \{z\})\cup \{x,y\}
\in\match_p(H)$, a contradiction.

We conclude that 
$\tau\cup\{x\}=\sigma\cup \{x,y\}\not\in \match_p(H)$.
So $\tau\in N$.

Thus, all elements unpaired after $M(x)\cup M(y)$ are paired as
$(\tau, \tau\cup\{z\})$.  That is, the three element pairings,
$M(x)$, $M(y)$, and $M(z)$ pair all faces of 
$\match_p(H_{1\times m\times n})$,
so $\match_p(H_{1\times m\times n})$ ($m,n\ge 3$) is contractible.
\end{proof}

\section{The $2 \times 2 \times 2$ honeycomb graph} 

We conclude this article by calculating the homotopy type of the perfect matching complex for the $2 \times 2 \times 2$ honeycomb graph. Recall that each perfect matching on the honeycomb graph is in bijection with a plane partition. For the $2 \times 2 \times 2$ honeycomb graph that means we are considering plane partitions of shape $(2,2)$. We will denote these plane partitions by 
\begin{tabular}{l|l}
$a$ & $b$ \\ \hline
$c$ & $d$
\end{tabular}
where $0 \leq a,b,c,d \leq 2 $ and $a \geq b,c \geq d$. For an example of the plane partition represented by 
\begin{tabular}{l|l}
$2$ & $2$ \\ \hline
$1$ & $1$
\end{tabular},
see Figure~\ref{tableau_example}.  See Appendix~A for all the perfect matchings of $H_{2\times 2\times 2}$ and their corresponding plane partitions.

Let $\sigma \in \match_p(H)$ be a face in the perfect matching complex of the $2 \times 2 \times 2$ honeycomb graph $H$. We use the notation $\sigma \in $
\begin{tabular}{l|l}
$a$ & $b$ \\ \hline
$c$ & $d$
\end{tabular}
to denote a subset of the perfect matching corresponding with \begin{tabular}{l|l}
$a$ & $b$ \\ \hline
$c$ & $d$
\end{tabular}. If an entry in the plane partition can be $0, 1,$ or $2$, in accordance with the restrictions, we will denote it with $\ast$. For example,  
\begin{tabular}{l|l}
$2$ & 2\\ \hline
$\ast$ & $1$
\end{tabular} 
represents the plane partitions 
\begin{tabular}{l|l}
$2$ & $2$ \\ \hline
$1$ & $1$
\end{tabular} 
and 
\begin{tabular}{l|l}
$2$ & $2$ \\ \hline
$2$ & $1$
\end{tabular}.

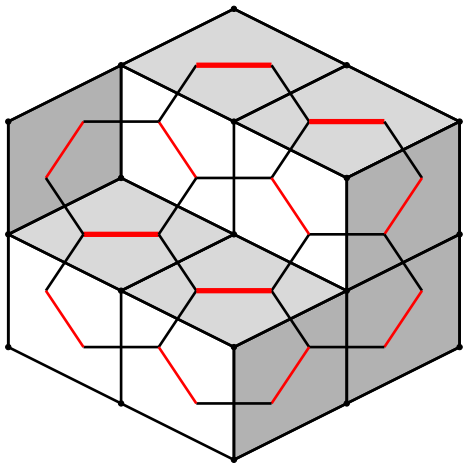
\begin{figure}[h]
\begin{center}
\begin{tikzpicture}[scale = 0.5]

\draw[line width = 1 pt, fill = gray!30] (7,12) -- (4,10.5) -- (7,9) -- (10, 10.5) -- (7, 12) ;

\draw[line width = 1 pt, fill = gray!30] (7,9) -- (10,10.5) -- (13,9) -- (10, 7.5) -- (7, 9) ;

\draw[line width = 1 pt, fill = gray!60] (13,9) -- (10,7.5) -- (10,4.5) -- (13, 6) -- (13, 9) ;

\draw[line width = 1 pt, fill = gray!60] (10,4.5) -- (13,6) -- (13,3) -- (10, 1.5) -- (10, 4.5) ;

\draw[line width = 1 pt, fill = gray!60] (10,1.5) -- (10,4.5) -- (7,3) -- (7, 0) -- (10, 1.5) ;

\draw[line width = 1 pt, fill = gray!30] (10,4.5) -- (7,3) -- (4,4.5) -- (7, 6) -- (10, 4.5) ;

\draw[line width = 1 pt, fill = gray!30] (4,4.5) -- (7,6) -- (4,7.5) -- (1, 6) -- (4, 4.5) ;

\draw[line width = 1 pt, fill = gray!60] (4,7.5) -- (1,6) -- (1,9) -- (4, 10.5) -- (4, 7.5) ;

\draw[red, line width = 1 pt] (2,7.5) -- (3,9);  
\draw[red, line width = 1 pt] (5,9) -- (6,7.5); 
\draw[line width = 1 pt] (6,10.5) -- (5,9);

\draw[red, line width = 2 pt] (6,10.5) -- (8,10.5);

\draw[line width = 1 pt] (8,10.5) -- (9,9);
\draw[line width = 1 pt] (9,9) -- (8,7.5);

\draw[red, line width = 2 pt] (9,9) -- (11,9);
\draw[line width = 1 pt] (12,7.5) -- (11,9);
\draw[red, line width = 1 pt] (12,7.5) -- (11,6);
\draw[line width = 1 pt] (9,6) -- (11,6);

\draw[line width = 1 pt] (11,6) -- (12,4.5);
\draw[red, line width = 1 pt] (12,4.5) -- (11,3);

\draw[line width = 1 pt] (11,3) -- (9,3);

\draw[line width = 1 pt] (6,7.5) -- (8,7.5); 

\draw[red, line width = 2 pt] (3,6) -- (5,6); 
\draw[line width = 1 pt] (5,6) -- (6,4.5); 
\draw[line width = 1 pt] (8,4.5) -- (9,6); 
\draw[line width = 1 pt] (8,4.5) -- (9,3); 
\draw[red, line width = 1 pt] (9,3) -- (8,1.5); 
\draw[line width = 1 pt] (8,1.5) -- (6,1.5);

\draw[line width = 1 pt] (5,3) -- (6,4.5);
\draw[line width = 1 pt] (3,3) -- (5,3);
\draw[line width = 1 pt] (2,4.5) -- (3,6);

\draw[line width = 1 pt] (3,9) -- (5,9); 
\draw[line width = 1 pt] (3,6) -- (2,7.5); 
\draw[red, line width = 1 pt] (3,3) -- (2, 4.5); 
\draw[line width = 1 pt] (5,6) -- (6, 7.5); 
\draw[red, line width = 1 pt] (9,6) -- (8, 7.5); 
\draw[red, line width = 2 pt] (6,4.5) -- (8, 4.5); 

\draw[red, line width = 1 pt] (6, 1.5) -- ( 5,3); 


\draw[line width = 1 pt] (7,12) -- (4,10.5) -- (7,9);
\draw[line width = 1 pt] (7,12) -- (10,10.5) -- (7,9);

\draw[line width = 1 pt] (7,9) -- (10,7.5) -- (13,9) -- (10,10.5);

\draw[line width = 1 pt] (13,9) -- (13,6) -- (10,4.5) -- (10, 7.5);
\draw[line width = 1 pt] (10,4.5) -- (10,1.5) -- (13, 3) -- (13, 6);
\draw[line width = 1 pt] (10,1.5) -- (7,0) -- (7,3) -- (10, 4.5);
\draw[line width = 1 pt] (7,0) -- (4,1.5) -- (4, 4.5) -- (7,3);
\draw[line width = 1 pt] (4,4.5) -- (7,6);
\draw[line width = 1 pt] (4,4.5) -- (1,6) -- (1,3) -- (4, 1.5);
\draw[line width = 1 pt] (1,3) -- (1,6) -- (4,7.5);
\draw[line width = 1 pt] (1,6) -- (1,9) -- (4, 10.5);

\draw[line width = 1 pt] (10,4.5) -- (7,6) -- (7, 9);
\draw[line width = 1 pt] (7,6) -- (4,7.5) -- (4, 10.5);

\filldraw[black] (7,12) circle (2pt);

\filldraw[black] (4, 10.5) circle (2pt);
\filldraw[black] (10, 10.5) circle (2pt);

\filldraw[black] (1,9) circle (2pt);
\filldraw[black] (7,9) circle (2pt);
\filldraw[black] (13, 9) circle (2pt);

\filldraw[black] (4,7.5) circle (2pt);
\filldraw[black] (10,7.5) circle (2pt);

\filldraw[black] (1, 6) circle (2pt);
\filldraw[black] (7, 6) circle (2pt);
\filldraw[black] (13, 6) circle (2pt);

\filldraw[black] (4, 4.5) circle (2pt);
\filldraw[black] (10, 4.5) circle (2pt);

\filldraw[black] (1, 3) circle (2pt);
\filldraw[black] (7, 3) circle (2pt);
\filldraw[black] (13,3) circle (2pt);

\filldraw[black] (4, 1.5) circle (2pt);
\filldraw[black] (10, 1.5) circle (2pt);

\filldraw[black] (7, 0) circle (2pt);

\end{tikzpicture}
\end{center}
\caption{Matching for plane partition 
\label{tableau_example}}
\end{figure}


\begin{remark}\label{rmk:observation} 
In the proof that follows we will use a sequence of element pairings to obtain the homotopy type.
Notice that if we perform a sequence of element pairings in which an element $\varepsilon$ has been paired on, we can categorize the remaining faces as those that contain $\varepsilon$ and those that do not. 
When we continue pairing with an element that has not previously been paired on, say $\lambda$, we can only pair faces of the same type. 
That is to say $\theta$ and $\theta \cup \lambda$ would be paired together only if $\varepsilon$ was in $\theta$ and $\theta \cup \lambda$, or $\varepsilon$ was not in $\theta$ and $\theta \cup \lambda$.  
\end{remark}

\begin{theorem} \label{thm:2by2by2}
Let $H$ be the honeycomb graph of dimension $2 \times 2 \times 2$. Then, 
\[ 
\match_p(H) \simeq S^3 \vee S^3.
\]
\end{theorem}

\begin{proof}
We will proceed by defining a discrete Morse matching on $\match_p(H)$ given by pairing on $\alpha$, then $\beta$, then $\gamma$, and finally $\delta$ according to the labels in Figure~\ref{fig:2by2by2proof}. 

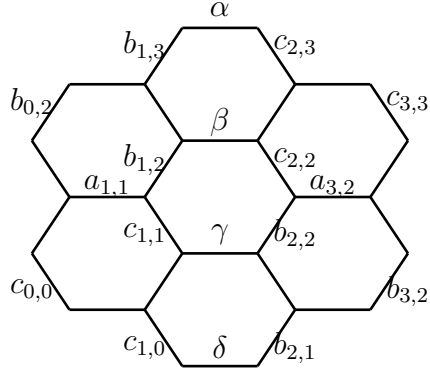
\begin{figure}[h]
\begin{center}
\begin{tikzpicture}[scale = 0.5]

\draw[line width = 1 pt] (2,7.5) -- (3,9);  
\draw[line width = 1 pt] (5,9) -- (6,7.5); 
\draw[line width = 1 pt] (6,10.5) -- (5,9);

\draw[line width = 1 pt] (6,10.5) -- (8,10.5);
\node at (7,11) {$\alpha$};

\draw[line width = 1 pt] (8,10.5) -- (9,9);
\draw[line width = 1 pt] (9,9) -- (8,7.5);

\draw[line width = 1 pt] (9,9) -- (11,9);
\draw[line width = 1 pt] (12,7.5) -- (11,9);
\draw[line width = 1 pt] (12,7.5) -- (11,6);
\draw[line width = 1 pt] (9,6) -- (11,6);

\draw[line width = 1 pt] (11,6) -- (12,4.5);
\draw[line width = 1 pt] (12,4.5) -- (11,3);

\draw[line width = 1 pt] (11,3) -- (9,3);

\draw[line width = 1 pt] (6,7.5) -- (8,7.5); 
\node at (7,8) {$\beta$};

\draw[line width = 1 pt] (3,6) -- (5,6); 
\draw[line width = 1 pt] (5,6) -- (6,4.5); 
\draw[line width = 1 pt] (8,4.5) -- (9,6); 
\draw[line width = 1 pt] (8,4.5) -- (9,3); 
\draw[line width = 1 pt] (9,3) -- (8,1.5); 
\draw[line width = 1 pt] (8,1.5) -- (6,1.5);
\node at (7,2) {$\delta$};

\draw[line width = 1 pt] (5,3) -- (6,4.5);
\draw[line width = 1 pt] (3,3) -- (5,3);
\draw[line width = 1 pt] (2,4.5) -- (3,6);

\draw[line width = 1 pt] (3,9) -- (5,9); 
\draw[line width = 1 pt] (3,6) -- (2,7.5); 
\draw[line width = 1 pt] (3,3) -- (2, 4.5); 
\draw[line width = 1 pt] (5,6) -- (6, 7.5); 
\draw[line width = 1 pt] (9,6) -- (8, 7.5); 
\draw[line width = 1 pt] (6,4.5) -- (8, 4.5); 
\node at (7,5) {$\gamma$};

\draw[line width = 1 pt] (6, 1.5) -- ( 5,3); 

\node at (5, 10) {$b_{1,3}$}; 
\node at (5, 7) {$b_{1,2}$}; 
\node at (5, 5) {$c_{1,1}$}; 
\node at (5, 2) {$c_{1,0}$};

\node at (2, 3.5) {$c_{0,0}$};
\node at (2, 8.5) {$b_{0,2}$};

\node at (9, 10) {$c_{2,3}$}; 
\node at (9, 7) {$c_{2,2}$}; 
\node at (9, 5) {$b_{2,2}$}; 
\node at (9, 2) {$b_{2,1}$};

\node at (12, 3.5) {$b_{3,2}$};
\node at (12, 8.5) {$c_{3,3}$};

\node at (4,6.27) {$a_{1,1}$};
\node at (10,6.27) {$a_{3,2}$};

\end{tikzpicture}
\end{center}
\caption{Figure for the proof of Theorem~\ref{thm:2by2by2}.} \label{fig:2by2by2proof} 
\end{figure}


For faces $\sigma \in \match_p(H)$, we make the following observations: 
\begin{itemize}
    \item[(1)] $\sigma \cup \{\alpha\}$ is a face if and only if $\sigma \in $
    \begin{tabular}{l|l}
$2$ & $\ast$\\ \hline
$\ast$ & $\ast$
\end{tabular}, 
\item[(2)] $\sigma \cup \{\beta\}$ is a face if and only if $\sigma \in$
\begin{tabular}{l|l}
$2$ & $2$\\ \hline
$2$ & $2$
\end{tabular} 
or $\sigma \in$
\begin{tabular}{l|l}
$1$ & $\ast$\\ \hline
$\ast$ & $\ast$
\end{tabular},
\item[(3)] $\sigma \cup \{\gamma\}$ is a face if and only if 
$\sigma \in$
\begin{tabular}{l|l}
$0$ & $0$\\ \hline
$0$ & $0$
\end{tabular} 
or $\sigma \in$
\begin{tabular}{l|l}
$\ast$ & $\ast$\\ \hline
$\ast$ & $1$
\end{tabular}, and
\item[(4)] $\sigma \cup \{\delta\}$ is a face if and only if 
$\sigma \in$
\begin{tabular}{l|l}
$\ast$ & $\ast$\\ \hline
$\ast$ & $0$
\end{tabular}.
\end{itemize}

We begin by an element pairing on $\alpha$ and then $\beta$. There are two types of unpaired simplices that remain. 
The first type (t1) consists of faces $\sigma \in \match_p(H)$ such that $\alpha, \beta \not \in \sigma$, and $\sigma$ cannot be paired with $\sigma \cup \alpha$ or $\sigma \cup \beta$,  because simplices $\sigma \cup \alpha$ and $\sigma \cup \beta$ do not exist in $\match_p(H)$. Therefore, by observations $(1)$ and $(2)$ $\sigma \not \in \begin{tabular}{l|l}
$2$ & $\ast$\\ \hline
$\ast$ & $\ast$ 
\end{tabular}$, 
$\sigma \not \in \begin{tabular}{l|l}
$2$ & $2$\\ \hline
$2$ & $2$ 
\end{tabular}$, and 
$\sigma \not \in \begin{tabular}{l|l}
$1$ & $\ast$\\ \hline
$\ast$ & $\ast$ 
\end{tabular}$. It follows that $\sigma$ is only in $\begin{tabular}{l|l}
$0$ & $0$\\ \hline
$0$ & $0$ 
\end{tabular}$ and $\alpha, \beta \not \in \sigma$.

The second type (t2) consists of faces $\sigma \cup \beta \in \match_p(H)$, which remain because $\sigma$ has been previously paired with $\sigma \cup \alpha$, and $\sigma \cup \beta \cup \alpha \not \in \match_p(H)$. 
There exists a face $\sigma \cup \alpha$, so $\sigma \in \begin{tabular}{l|l}
$2$ & $\ast$\\ \hline
$\ast$ & $\ast$ 
\end{tabular}$. Similarly, $\sigma \cup \beta$ is a face, so $\sigma \in \begin{tabular}{l|l}
$1$ & $\ast$\\ \hline
$\ast$ & $\ast$ 
\end{tabular}$ or $\sigma \in \begin{tabular}{l|l}
$2$ & $2$\\ \hline
$2$ & $2$ 
\end{tabular}$. But $\sigma \cup \beta \cup \alpha$ is not a face, so $\sigma$ is not in $\begin{tabular}{l|l}
$2$ & $2$\\ \hline
$2$ & $2$ 
\end{tabular}$. Therefore, $(t2)$ is the set of faces $\sigma \cup \beta$ such that $\alpha,\beta \not \in \sigma$, $\sigma \in \begin{tabular}{l|l}
$2$ & $\ast$\\ \hline
$\ast$ & $\ast$ 
\end{tabular}$ 
and 
$\sigma \in \begin{tabular}{l|l}
$1$ & $\ast$\\ \hline
$\ast$ & $\ast$ 
\end{tabular}$, and $\sigma \not \in \begin{tabular}{l|l}
$2$ & $2$\\ \hline
$2$ & $2$ 
\end{tabular}$.

By Remark~\ref{rmk:observation}, we know that in the remaining element pairings if two faces are paired they have to be of the same type. We perform our next pairing  with $\gamma$ and analyze what faces remain unpaired. For each of the above types we will consider faces that contain $\gamma$ and those that do not.  

Let $\theta \in \match_p(H)$. 

\textbf{Case 1: Suppose $\theta$ is of type $(t1)$ and $\gamma \not \in \theta$.} Since $\theta$ is type $(t1)$, $\theta$ is only in 
$\begin{tabular}{l|l}
$0$ & $0$\\ \hline
$0$ & $0$ 
\end{tabular}$
but, it is also the case that $\gamma \in 
\begin{tabular}{l|l}
$0$ & $0$\\ \hline
$0$ & $0$ 
\end{tabular}$. 
This means that $\theta \cup \gamma$ is only in 
\begin{tabular}{l|l}
$0$ & $0$\\ \hline
$0$ & $0$ 
\end{tabular}, so $\theta$ and $\theta \cup \gamma$ are paired and there are no unpaired faces that remain from this case. 

\textbf{Case 2: Suppose $\theta \cup \gamma$ is of type $(t1)$.} Then $\theta \cup \gamma$ is only in 
$\begin{tabular}{l|l}
$0$ & $0$\\ \hline
$0$ & $0$ 
\end{tabular}$. The only way this face could be unpaired  is if $\theta$ is not only in 
\begin{tabular}{l|l}
$0$ & $0$\\ \hline
$0$ & $0$ 
\end{tabular}. 
Hence the faces that remain are all $\theta \cup \gamma$ such that $\alpha, \beta, \gamma \not \in \theta$ and $\theta \cup \gamma$ is only in $\begin{tabular}{l|l}
$0$ & $0$\\ \hline
$0$ & $0$ 
\end{tabular}$ and either
$\theta \in 
\begin{tabular}{l|l}
$2$ & $\ast$\\ \hline
$\ast$ & $\ast$ 
\end{tabular}$ or 
$\theta \in 
\begin{tabular}{l|l}
$1$ & $\ast$\\ \hline
$\ast$ & $\ast$ 
\end{tabular}$. We call these faces types $(t1.1)$. 

\textbf{Case 3: Suppose $\theta$ is of type $(t2)$ and $\gamma \not \in \theta$.} Since $\theta$ is of type $(t2)$ $\theta = \sigma \cup \beta$ where $\alpha, \beta \not \in \sigma$ and $\sigma \in 
\begin{tabular}{l|l}
$2$ & $\ast$\\ \hline
$\ast$ & $\ast$ 
\end{tabular}$, and 
$\sigma \in 
\begin{tabular}{l|l}
$1$ & $\ast$\\ \hline
$\ast$ & $\ast$ 
\end{tabular}$, and 
$\sigma \not \in 
\begin{tabular}{l|l}
$2$ & $2$\\ \hline
$2$ & $2$ 
\end{tabular}$. Since we are supposing that $\theta$ is unpaired, $\theta \cup \gamma = \sigma \cup \gamma \cup \beta$ is not of type $(t2)$.
Therefore,  $\sigma \cup \gamma \not \in 
\begin{tabular}{l|l}
$2$ & $\ast$\\ \hline
$\ast$ & $\ast$ 
\end{tabular}$ or 
$\sigma \cup \gamma \not \in 
\begin{tabular}{l|l}
$1$ & $\ast$\\ \hline
$\ast$ & $\ast$ 
\end{tabular}$ 
or 
$\sigma \cup \gamma \in 
\begin{tabular}{l|l}
$2$ & $2$\\ \hline
$2$ & $2$ 
\end{tabular}$. Notice the last condition is not possible since $\gamma \not \in 
\begin{tabular}{l|l}
$2$ & $2$\\ \hline
$2$ & $2$ 
\end{tabular}$ and, since $\gamma$ is in a face if and only if the face is in 
$\begin{tabular}{l|l}
$\ast$ & $\ast$\\ \hline
$\ast$ & $1$ 
\end{tabular}$ or in 
\begin{tabular}{l|l}
$0$ & $0$\\ \hline
$0$ & $0$ 
\end{tabular}, we can rewrite the above statement. That is, since $\theta \cup \gamma$ is not of type $(t2)$,  
$\sigma \not \in 
\begin{tabular}{l|l}
$2$ & $\ast$\\ \hline
$\ast$ & $1$ 
\end{tabular}$ or $\sigma \not \in 
\begin{tabular}{l|l}
$1$ & $1$\\ \hline
$1$ & $1$ 
\end{tabular}$. 
Hence the unpaired faces that remain are $\sigma \cup \beta$ where $ \alpha, \beta, \gamma \not \in \sigma$, $\sigma \in 
\begin{tabular}{l|l}
$2$ & $\ast$\\ \hline
$\ast$ & $\ast$ 
\end{tabular}$
and 
$\sigma \in 
\begin{tabular}{l|l}
$1$ & $\ast$\\ \hline
$\ast$ & $\ast$ 
\end{tabular}$
and $\sigma \not \in 
\begin{tabular}{l|l}
$2$ & $2$\\ \hline
$2$ & $2$ 
\end{tabular}$
and either $\sigma$ is not in $
\begin{tabular}{l|l}
$2$ & $\ast$\\ \hline
$\ast$ & $1$ 
\end{tabular}$
or $\sigma$ is not in 
\begin{tabular}{l|l}
$1$ & $1$\\ \hline
$1$ & $1$ 
\end{tabular}. 
We call this type $(t2.1)$.

\textbf{Case 4: Suppose $\theta \cup \gamma$ is of type $(t2)$.} Then, $\theta \cup \gamma = \sigma \cup \beta \cup \gamma$ and, as argued in Case 3, $\sigma \in 
\begin{tabular}{l|l}
$2$ & $\ast$\\ \hline
$\ast$ & $1$ 
\end{tabular}$ 
and 
$\sigma \in 
\begin{tabular}{l|l}
$1$ & $1$\\ \hline
$1$ & $1$ 
\end{tabular}$. Since we are assuming that $\theta \cup \gamma = \sigma \cup \gamma \cup \beta$ is unpaired this must be because $\theta = \sigma \cup \beta$ is not of type $(t2)$. This implies that $\theta \not \in 
\begin{tabular}{l|l}
$2$ & $\ast$\\ \hline
$\ast$ & $\ast$ 
\end{tabular}$, which cannot be the case because $\sigma \cup \gamma \in 
\begin{tabular}{l|l}
$2$ & $\ast$\\ \hline
$\ast$ & $1$ 
\end{tabular}$, or 
$\sigma \not \in 
\begin{tabular}{l|l}
$1$ & $\ast$\\ \hline
$\ast$ & $\ast$ 
\end{tabular}$, which cannot be the case because $\sigma \cup \gamma \in 
\begin{tabular}{l|l}
$1$ & $1$\\ \hline
$1$ & $1$ 
\end{tabular}$, or 
$\sigma \in 
\begin{tabular}{l|l}
$2$ & $2$\\ \hline
$2$ & $2$ 
\end{tabular}$. It follows that, for this case, all unpaired faces are $\sigma \cup \beta \cup \gamma$ such that $\alpha, \beta, \gamma \not \in \sigma$ and $\sigma \in 
\begin{tabular}{l|l}
$2$ & $\ast$\\ \hline
$\ast$ & $1$ 
\end{tabular}$ and 
$\sigma \in 
\begin{tabular}{l|l}
$1$ & $1$\\ \hline
$1$ & $1$ 
\end{tabular}$
and 
$\sigma \in 
\begin{tabular}{l|l}
$2$ & $2$\\ \hline
$2$ & $2$ 
\end{tabular}$. We call this type $(t2.2)$.

We are now ready to pair using $\delta$. Consider first the faces of type $(t2.2)$. 

\textbf{Case 4.1: Suppose $\theta$ is of type $(t2.2)$ and $\delta \in \theta$.}  Then $\delta \in \sigma.$ This case cannot occur because in type (t2.2) 
$\sigma \in 
\begin{tabular}{l|l}
$1$ & $1$\\ \hline
$1$ & $1$ 
\end{tabular}$
and $\delta \not\in \begin{tabular}{l|l}
$1$ & $1$\\ \hline
$1$ & $1$ 
\end{tabular}. $

\textbf{Case 4.2: Suppose $\theta$ is of type $(t2.2)$ and $\delta \not \in \theta$.} That is, $\theta = \sigma \cup \beta \cup 
\gamma$ such that $\alpha, \beta, \gamma, \delta \not \in \sigma$ and 
$\sigma \in 
\begin{tabular}{l|l}
$2$ & $\ast$\\ \hline
$\ast$ & $1$ 
\end{tabular}$ and
$\sigma \in 
\begin{tabular}{l|l}
$1$ & $1$\\ \hline
$1$ & $1$ 
\end{tabular}$
and 
$\sigma \in 
\begin{tabular}{l|l}
$2$ & $2$\\ \hline
$2$ & $2$ 
\end{tabular}$. For the faces left unpaired, it must be the case that $\sigma \cup \delta \cup \beta \cup \gamma$ is not of type $(t2.2)$. That is, $\sigma \cup \delta \not \in 
\begin{tabular}{l|l}
$2$ & $\ast$\\ \hline
$\ast$ & $1$ 
\end{tabular}$ or
$\sigma \cup \delta \not \in 
\begin{tabular}{l|l}
$1$ & $1$\\ \hline
$1$ & $1$ 
\end{tabular}$
or  
$\sigma \cup \delta \not \in 
\begin{tabular}{l|l}
$2$ & $2$\\ \hline
$2$ & $2$ 
\end{tabular}$, which is true since $\delta \not\in \begin{tabular}{l|l}
$1$ & $1$\\ \hline
$1$ & $1$ 
\end{tabular}$. 
Therefore, the unpaired cells are $\sigma \cup \beta \cup \gamma$ such that $ \alpha, \beta, \gamma, \delta \not \in \sigma$ and 
$\sigma \in 
\begin{tabular}{l|l}
$2$ & $\ast$\\ \hline
$\ast$ & $1$ 
\end{tabular}$ and
$\sigma \in 
\begin{tabular}{l|l}
$1$ & $1$\\ \hline
$1$ & $1$ 
\end{tabular}$
and 
$\sigma \in 
\begin{tabular}{l|l}
$2$ & $2$\\ \hline
$2$ & $2$ 
\end{tabular}$. When we consider the intersection 
$ 
\begin{tabular}{l|l}
$2$ & $\ast$\\ \hline
$\ast$ & $1$ 
\end{tabular}$
$\bigcap 
\begin{tabular}{l|l}
$1$ & $1$\\ \hline
$1$ & $1$ 
\end{tabular}$
$\bigcap 
\begin{tabular}{l|l}
$2$ & $2$\\ \hline
$2$ & $2$ 
\end{tabular}$ 
we have $c_{0,0}, c_{1,0}, b_{2,1},$ and $b_{3,2}$. See Figure~\ref{fig:first_intersection}. 
These four edges are in the intersection
$ 
\begin{tabular}{l|l}
$2$ & $2$\\ \hline
$2$ & $1$ 
\end{tabular}$
$\bigcap 
\begin{tabular}{l|l}
$1$ & $1$\\ \hline
$1$ & $1$ 
\end{tabular}$
$\bigcap 
\begin{tabular}{l|l}
$2$ & $2$\\ \hline
$2$ & $2$ 
\end{tabular}$.
One can verify that these are the only edges in the intersection 
$ 
\begin{tabular}{l|l}
$2$ & $\ast$\\ \hline
$\ast$ & $1$ 
\end{tabular}$
$\bigcap 
\begin{tabular}{l|l}
$1$ & $1$\\ \hline
$1$ & $1$ 
\end{tabular}$
$\bigcap 
\begin{tabular}{l|l}
$2$ & $2$\\ \hline
$2$ & $2$ 
\end{tabular}$.

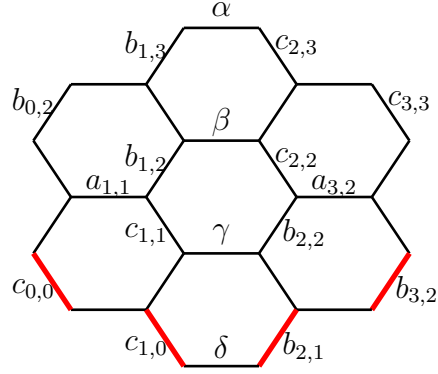
\begin{figure}[h]
\begin{center}
\begin{tikzpicture}[scale = 0.5]

\draw[line width = 1 pt] (2,7.5) -- (3,9);  
\draw[line width = 1 pt] (5,9) -- (6,7.5); 
\draw[line width = 1 pt] (6,10.5) -- (5,9);

\draw[line width = 1 pt] (6,10.5) -- (8,10.5);
\node at (7,11) {$\alpha$};

\draw[line width = 1 pt] (8,10.5) -- (9,9);
\draw[line width = 1 pt] (9,9) -- (8,7.5);

\draw[line width = 1 pt] (9,9) -- (11,9);
\draw[line width = 1 pt] (12,7.5) -- (11,9);
\draw[line width = 1 pt] (12,7.5) -- (11,6);
\draw[line width = 1 pt] (9,6) -- (11,6);

\draw[line width = 1 pt] (11,6) -- (12,4.5);
\draw[line width = 2 pt, red] (12,4.5) -- (11,3);

\draw[line width = 1 pt] (11,3) -- (9,3);

\draw[line width = 1 pt] (6,7.5) -- (8,7.5); 
\node at (7,8) {$\beta$};

\draw[line width = 1 pt] (3,6) -- (5,6); 
\draw[line width = 1 pt] (5,6) -- (6,4.5); 
\draw[line width = 1 pt] (8,4.5) -- (9,6); 
\draw[line width = 1 pt] (8,4.5) -- (9,3); 
\draw[line width = 2 pt, red] (9,3) -- (8,1.5); 
\draw[line width = 1 pt] (8,1.5) -- (6,1.5);
\node at (7,2) {$\delta$};

\draw[line width = 1 pt] (5,3) -- (6,4.5);
\draw[line width = 1 pt] (3,3) -- (5,3);
\draw[line width = 1 pt] (2,4.5) -- (3,6);

\draw[line width = 1 pt] (3,9) -- (5,9); 
\draw[line width = 1 pt] (3,6) -- (2,7.5); 
\draw[line width = 2 pt, red] (3,3) -- (2, 4.5); 
\draw[line width = 1 pt] (5,6) -- (6, 7.5); 
\draw[line width = 1 pt] (9,6) -- (8, 7.5); 
\draw[line width = 1 pt] (6,4.5) -- (8, 4.5); 
\node at (7,5) {$\gamma$};

\draw[line width = 2 pt, red] (6, 1.5) -- (5,3); 

\node at (5, 10) {$b_{1,3}$}; 
\node at (5, 7) {$b_{1,2}$}; 
\node at (5, 5) {$c_{1,1}$}; 
\node at (5, 2) {$c_{1,0}$};

\node at (2, 3.5) {$c_{0,0}$};
\node at (2, 8.5) {$b_{0,2}$};

\node at (9, 10) {$c_{2,3}$}; 
\node at (9, 7) {$c_{2,2}$}; 
\node at (9.2, 5) {$b_{2,2}$}; 
\node at (9.2, 2) {$b_{2,1}$};

\node at (12.2, 3.5) {$b_{3,2}$};
\node at (12, 8.5) {$c_{3,3}$};

\node at (4,6.27) {$a_{1,1}$};
\node at (10,6.27) {$a_{3,2}$};

\end{tikzpicture}
\end{center}
\caption{The highlighted edges used in Case 4.2.}
\label{fig:first_intersection}
\end{figure}


Therefore, $\sigma \in \mathcal{P}(c_{0,0},c_{1,0}, b_{2,1}, b_{3,2})$, the 
power set of $\{c_{0,0}, c_{1,0}, b_{2,1}, b_{3,2}\}$.  So the faces $\theta$ in Case~4.2 form an interval in the face poset of $\match_p(H)$ isomorphic to the Boolean lattice, and so can be paired using an element pairing with, say, $c_{1,0}$. 

We proceed with the analysis of element pairing using $\delta$ for unpaired faces of type (t1.1) from Case 2.

\textbf{Case 2.1: Suppose $\theta$ is of type $(t1.1)$ and $\delta \not \in \theta$.} 
That is, $\theta = \sigma \cup \gamma$ such that $\alpha, \beta, \gamma, \delta \not \in \sigma$, $\sigma \cup \gamma$ is only in 
$\begin{tabular}{l|l}
$0$ & $0$\\ \hline
$0$ & $0$ 
\end{tabular}$ and $\sigma \in 
\begin{tabular}{l|l}
$2$ & $\ast$\\ \hline
$\ast$ & $\ast$ 
\end{tabular}$ or 
$\sigma \in 
\begin{tabular}{l|l}
$1$ & $\ast$\\ \hline
$\ast$ & $\ast$ 
\end{tabular}$.
For the faces left unpaired, it must be that $\sigma \cup \delta \cup \gamma$ is not of type $(t1.1)$.
Therefore, it is either the case that $\sigma \cup \delta \cup \gamma$ is not only in $
\begin{tabular}{l|l}
$0$ & $0$\\ \hline
$0$ & $0$ 
\end{tabular}$ or that 
$\sigma \cup \delta \not \in
\begin{tabular}{l|l}
$2$ & $\ast$\\ \hline
$\ast$ & $\ast$ 
\end{tabular}$ and 
$\sigma \cup \delta \not \in
\begin{tabular}{l|l}
$1$ & $\ast$\\ \hline
$\ast$ & $\ast$ 
\end{tabular}$, but notice if 
$\sigma \cup \delta \cup \gamma$ is not only in $ 
\begin{tabular}{l|l}
$0$ & $0$\\ \hline
$0$ & $0$ 
\end{tabular}$ then 
$\sigma \cup \gamma$ is not only in  $
\begin{tabular}{l|l}
$0$ & $0$\\ \hline
$0$ & $0$ 
\end{tabular}$, which is a contradiction to $\theta$ being type $(t1.1)$. 
Thus, it must be the case that 
$\sigma \cup \delta \not \in
\begin{tabular}{l|l}
$2$ & $\ast$\\ \hline
$\ast$ & $\ast$ 
\end{tabular}$ and 
$\sigma \cup \delta \not \in
\begin{tabular}{l|l}
$1$ & $\ast$\\ \hline
$\ast$ & $\ast$ 
\end{tabular}$ and, in particular,
$\sigma \not \in
\begin{tabular}{l|l}
$2$ & $\ast$\\ \hline
$\ast$ & $0$ 
\end{tabular}$ and 
$\sigma \not \in
\begin{tabular}{l|l}
$1$ & $\ast$\\ \hline
$\ast$ & $0$ 
\end{tabular}$. It follows that the faces left unpaired are those such that $\sigma \cup \gamma$ is only in 
$\begin{tabular}{l|l}
$0$ & $0$\\ \hline
$0$ & $0$ 
\end{tabular}$
and 
$ ( \sigma \in 
\begin{tabular}{l|l}
$2$ & $2$\\ \hline
$2$ & $2$ 
\end{tabular}$ or 
$\sigma \in 
\begin{tabular}{l|l}
$2$ & $\ast$\\ \hline
$\ast$ & $1$ 
\end{tabular}$ or 
$\sigma \in 
\begin{tabular}{l|l}
$1$ & $1$\\ \hline
$1$ & $1$ 
\end{tabular} )$. 
Notice it cannot be the case that 
$\sigma \in 
\begin{tabular}{l|l}
$2$ & $\ast$\\ \hline
$\ast$ & $1$ 
\end{tabular}$ or 
$\sigma \in 
\begin{tabular}{l|l}
$1$ & $1$\\ \hline
$1$ & $1$ 
\end{tabular}$ 
because 
$\sigma \cup \gamma$ is only in 
$\begin{tabular}{l|l}
$0$ & $0$\\ \hline
$0$ & $0$ 
\end{tabular}$. 
So we are left with 
$\sigma \cup \gamma$ only in 
$\begin{tabular}{l|l}
$0$ & $0$\\ \hline
$0$ & $0$ 
\end{tabular}$
and 
$\sigma \in 
\begin{tabular}{l|l}
$2$ & $2$\\ \hline
$2$ & $2$ 
\end{tabular}$, 
but this implies that $\sigma$ is an empty face,
 because $\begin{tabular}{l|l}
$0$ & $0$\\ \hline
$0$ & $0$ 
\end{tabular} \cap \begin{tabular}{l|l}
$2$ & $2$\\ \hline
$2$ & $2$ 
\end{tabular}= \{\emptyset\}.$
Therefore, $\sigma \cup \gamma = \gamma$, and this is a contradiction because $\gamma$ is not only in 
$\begin{tabular}{l|l}
$0$ & $0$\\ \hline
$0$ & $0$ 
\end{tabular}$. Hence, there are no unpaired faces for this case. 

\textbf{Case 2.2: Suppose $\theta$ is of type $(t1.1)$ and $\delta \in \theta$.} That is, $\theta = \sigma \cup \delta \cup \gamma$ where $\alpha, \beta, \gamma, \delta \not \in \sigma$, $\sigma \cup \delta \cup \gamma$ is only in 
$\begin{tabular}{l|l}
$0$ & $0$\\ \hline
$0$ & $0$ 
\end{tabular}$ and 
 $\sigma \cup \delta \in 
\begin{tabular}{l|l}
$2$ & $\ast$\\ \hline
$\ast$ & $\ast$ 
\end{tabular}$ or 
$\sigma \cup \delta \in 
\begin{tabular}{l|l}
$1$ & $\ast$\\ \hline
$\ast$ & $\ast$ 
\end{tabular}$.
For the faces left unpaired, it must be that $\sigma \cup \gamma$ is not of type $(t1.1)$. Therefore, it is either the case that $\sigma \cup \gamma$ is not only in 
$\begin{tabular}{l|l}
$0$ & $0$\\ \hline
$0$ & $0$ 
\end{tabular}$ 
or 
$\sigma \not \in 
\begin{tabular}{l|l}
$2$ & $\ast$\\ \hline
$\ast$ & $\ast$ 
\end{tabular}$
and
$\sigma \not \in 
\begin{tabular}{l|l}
$1$ & $\ast$\\ \hline
$\ast$ & $\ast$ 
\end{tabular}$, but $\sigma \cup \delta$ 
is already in 
$
\begin{tabular}{l|l}
$2$ & $\ast$\\ \hline
$\ast$ & $\ast$ 
\end{tabular}$ or 
$ 
\begin{tabular}{l|l}
$1$ & $\ast$\\ \hline
$\ast$ & $\ast$ 
\end{tabular}$. 
Therefore, it must be that 
$\sigma \cup \gamma$ is not only in 
$\begin{tabular}{l|l}
$0$ & $0$\\ \hline
$0$ & $0$ 
\end{tabular}$.
Hence we have that all unpaired faces of this type are such that $\sigma \in 
$
$\begin{tabular}{l|l}
$0$ & $0$\\ \hline
$0$ & $0$ 
\end{tabular}$ 
and 
($\sigma \in 
\begin{tabular}{l|l}
$2$ & $\ast$\\ \hline
$\ast$ & $0$ 
\end{tabular}$ or
$\sigma \in 
\begin{tabular}{l|l}
$1$ & $\ast$\\ \hline
$\ast$ & $0$ 
\end{tabular}$) 
and 
($\sigma \in 
\begin{tabular}{l|l}
$2$ & $\ast$\\ \hline
$\ast$ & $1$ 
\end{tabular}$ or
$\sigma \in 
\begin{tabular}{l|l}
$1$ & $1$\\ \hline
$1$ & $1$ 
\end{tabular}$). 

These matchings $\sigma$ are subsets of the highlighted edges
in Figure~\ref{fig:second_intersection}.
It can be checked that no other edges are in a matching $\sigma$ of this type.  
The edges $b_{0,2}$, $b_{1,3}$, $c_{2,3}$, and $c_{3,3}$ are all in 
$\begin{tabular}{l|l}
$0$ & $0$\\ \hline
$0$ & $0$ 
\end{tabular} \bigcap 
\begin{tabular}{l|l}
$1$ & $1$\\ \hline
$1$ & $0$ 
\end{tabular}  \bigcap
\begin{tabular}{l|l}
$1$ & $1$\\ \hline
$1$ & $1$ 
\end{tabular}$. 
So the unpaired faces $\sigma\cup \delta \cup \gamma$ include those
for which $\sigma\in P(b_{0,2}, b_{1,3}, c_{2,3}, c_{3,3})$.
The edges $b_{1,2}$ and $c_{2,2}$ are both in 
$\begin{tabular}{l|l}
$0$ & $0$\\ \hline
$0$ & $0$ 
\end{tabular} \bigcap 
\begin{tabular}{l|l}
$2$ & $2$\\ \hline
$2$ & $0$ 
\end{tabular}  \bigcap
\begin{tabular}{l|l}
$2$ & $2$\\ \hline
$2$ & $1$ 
\end{tabular}$. 
The edges $b_{1,2}$ and $c_{3,3}$ are both in 
$\begin{tabular}{l|l}
$0$ & $0$\\ \hline
$0$ & $0$ 
\end{tabular} \bigcap 
\begin{tabular}{l|l}
$2$ & $0$\\ \hline
$2$ & $0$ 
\end{tabular}  \bigcap
\begin{tabular}{l|l}
$2$ & $1$\\ \hline
$2$ & $1$ 
\end{tabular}$. 
The edges $b_{0,2}$ and $c_{2,2}$ are both in 
$\begin{tabular}{l|l}
$0$ & $0$\\ \hline
$0$ & $0$ 
\end{tabular} \bigcap 
\begin{tabular}{l|l}
$2$ & $2$\\ \hline
$0$ & $0$ 
\end{tabular}  \bigcap
\begin{tabular}{l|l}
$2$ & $2$\\ \hline
$1$ & $1$ 
\end{tabular}$. 
So the unpaired faces $\sigma\cup \delta \cup \gamma$ include those
for which 
$\sigma\in  \{ \{b_{1,2}\}, \{c_{2,2}\}, \{b_{1,2}, c_{2,2}\}, \{b_{1,2}, c_{3,3}\}, \{b_{0,2}, c_{2,2}\}\}$.
It is straightforward to check that these describe all the unpaired faces.
Thus, the unpaired faces in this case are $\sigma \cup 
\delta \cup \gamma$ where $\sigma \in P(b_{0,2}, b_{1,3}, c_{3,3}, c_{2,3}) \cup
 \{ \{b_{1,2}\}, \{c_{2,2}\}, \{b_{1,2}, c_{2,2}\}, \{b_{1,2}, c_{3,3}\}, \{b_{0,2}, c_{2,2}\}\}$.


\begin{figure}[h]
\begin{center}
\begin{tikzpicture}[scale = 0.5]

\draw[line width = 2 pt, red] (2,7.5) -- (3,9);  
\draw[line width = 1 pt] (5,9) -- (6,7.5); 
\draw[line width = 2 pt, red] (6,10.5) -- (5,9);

\draw[line width = 1 pt] (6,10.5) -- (8,10.5);
\node at (7,11) {$\alpha$};

\draw[line width = 2 pt, red] (8,10.5) -- (9,9);
\draw[line width = 1 pt] (9,9) -- (8,7.5);

\draw[line width = 1 pt] (9,9) -- (11,9);
\draw[line width = 2 pt, red] (12,7.5) -- (11,9);
\draw[line width = 1 pt] (12,7.5) -- (11,6);
\draw[line width = 1 pt] (9,6) -- (11,6);

\draw[line width = 1 pt] (11,6) -- (12,4.5);
\draw[line width = 1 pt] (12,4.5) -- (11,3);

\draw[line width = 1 pt] (11,3) -- (9,3);

\draw[line width = 1 pt] (6,7.5) -- (8,7.5); 
\node at (7,8) {$\beta$};

\draw[line width = 1 pt] (3,6) -- (5,6); 
\draw[line width = 1 pt] (5,6) -- (6,4.5); 
\draw[line width = 1 pt] (8,4.5) -- (9,6); 
\draw[line width = 1 pt] (8,4.5) -- (9,3); 
\draw[line width = 1 pt] (9,3) -- (8,1.5); 
\draw[line width = 1 pt] (8,1.5) -- (6,1.5);
\node at (7,2) {$\delta$};

\draw[line width = 1 pt] (5,3) -- (6,4.5);
\draw[line width = 1 pt] (3,3) -- (5,3);
\draw[line width = 1 pt] (2,4.5) -- (3,6);

\draw[line width = 1 pt] (3,9) -- (5,9); 
\draw[line width = 1 pt] (3,6) -- (2,7.5); 
\draw[line width = 1 pt] (3,3) -- (2, 4.5); 
\draw[line width = 2 pt, red] (5,6) -- (6, 7.5); 
\draw[line width = 2 pt, red] (9,6) -- (8, 7.5); 
\draw[line width = 1 pt] (6,4.5) -- (8, 4.5); 
\node at (7,5) {$\gamma$};

\draw[line width = 1 pt] (6, 1.5) -- (5,3); 

\node at (5, 10) {$b_{1,3}$}; 
\node at (5, 7) {$b_{1,2}$}; 
\node at (5, 5) {$c_{1,1}$}; 
\node at (5, 2) {$c_{1,0}$};

\node at (2, 3.5) {$c_{0,0}$};
\node at (2, 8.5) {$b_{0,2}$};

\node at (9, 10) {$c_{2,3}$}; 
\node at (9, 7) {$c_{2,2}$}; 
\node at (9.2, 5) {$b_{2,2}$}; 
\node at (9.2, 2) {$b_{2,1}$};

\node at (12.2, 3.5) {$b_{3,2}$};
\node at (12, 8.5) {$c_{3,3}$};

\node at (4,6.27) {$a_{1,1}$};
\node at (10,6.27) {$a_{3,2}$};

\end{tikzpicture}
\end{center}
\caption{The highlighted edges used in Case 2.2} \label{fig:second_intersection}
\end{figure}
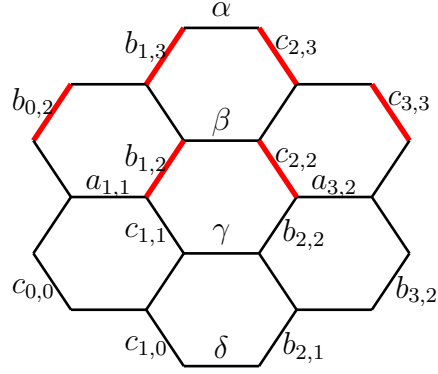


Finally, we discuss Case 3 with element pairing using $\delta$.

\textbf{Case 3.1: Suppose $\theta$ is of type $(t2.1)$ and $\delta \not \in \theta$.} That is, $\theta = \sigma \cup \beta$ where $\alpha, \beta, \gamma, \delta \not \in \sigma$, $\sigma \in 
\begin{tabular}{l|l}
$2$ & $\ast$\\ \hline
$\ast$ & $\ast$ 
\end{tabular}$, 
$ \sigma \in 
\begin{tabular}{l|l}
$1$ & $\ast$\\ \hline
$\ast$ & $\ast$ 
\end{tabular}$, 
$\sigma \not \in 
\begin{tabular}{l|l}
$2$ & $2$\\ \hline
$2$ & $2$ 
\end{tabular}$, and 
($\sigma \not \in 
\begin{tabular}{l|l}
$2$ & $\ast$\\ \hline
$\ast$ & $1$ 
\end{tabular}$ 
or 
$\sigma \not \in 
\begin{tabular}{l|l}
$1$ & $1$\\ \hline
$1$ & $1$ 
\end{tabular}$).
If $\theta\cup \delta=\sigma\cup\delta\cup\beta$ is also of type $(t2.1)$, 
then $\theta$ and $\theta\cup \delta$ are paired using $\delta$.  
So assume $\sigma\cup\delta\cup\beta$ is not of type $(t2.1)$.
Then
 $\sigma \cup \delta \not \in 
\begin{tabular}{l|l}
$2$ & $\ast$\\ \hline
$\ast$ & $\ast$ 
\end{tabular}$
or
$ \sigma \cup \delta \not \in 
\begin{tabular}{l|l}
$1$ & $\ast$\\ \hline
$\ast$ & $\ast$ 
\end{tabular}$ 
or
$\sigma \cup \delta \in 
\begin{tabular}{l|l}
$2$ & $2$\\ \hline
$2$ & $2$ 
\end{tabular}$
or
($\sigma \cup \delta \in 
\begin{tabular}{l|l}
$2$ & $\ast$\\ \hline
$\ast$ & $1$ 
\end{tabular}$ 
and 
$\sigma \cup \delta \in 
\begin{tabular}{l|l}
$1$ & $1$\\ \hline
$1$ & $1$ 
\end{tabular}$).
Note that by observation (4) $\sigma\cup\delta$ cannot be in 
$\begin{tabular}{l|l}
$2$ & $2$\\ \hline
$2$ & $2$ 
\end{tabular}$ or 
$\begin{tabular}{l|l}
$1$ & $1$\\ \hline
$1$ & $1$ 
\end{tabular}$.
Therefore, we see that all unpaired faces are such that $\sigma \in 
\begin{tabular}{l|l}
$2$ & $\ast$\\ \hline
$\ast$ & $\ast$ 
\end{tabular}$, 
and 
$\sigma \in 
\begin{tabular}{l|l}
$1$ & $\ast$\\ \hline
$\ast$ & $\ast$ 
\end{tabular}$
and 
$\sigma \not \in 
\begin{tabular}{l|l}
$2$ & $2$\\ \hline
$2$ & $2$ 
\end{tabular}$, and 
($\sigma \not \in 
\begin{tabular}{l|l}
$2$ & $\ast$\\ \hline
$\ast$ & $1$ 
\end{tabular}$ 
or 
$\sigma \not \in 
\begin{tabular}{l|l}
$1$ & $1$\\ \hline
$1$ & $1$ 
\end{tabular}$) 
and 
($\sigma \not \in 
\begin{tabular}{l|l}
$2$ & $\ast$\\ \hline
$\ast$ & $0$ 
\end{tabular}$
or 
$\sigma \not \in 
\begin{tabular}{l|l}
$1$ & $\ast$\\ \hline
$\ast$ & $0$ 
\end{tabular}$). Although there are four possible subcases to consider, we see that it is not possible for $\sigma \in 
\begin{tabular}{l|l}
$2$ & $\ast$\\ \hline
$\ast$ & $\ast$ 
\end{tabular}$, 
$\sigma \not \in 
\begin{tabular}{l|l}
$2$ & $2$\\ \hline
$2$ & $2$ 
\end{tabular}$, 
$\sigma \not \in 
\begin{tabular}{l|l}
$2$ & $\ast$\\ \hline
$\ast$ & $1$ 
\end{tabular}$, 
and $\sigma \not \in 
\begin{tabular}{l|l}
$2$ & $\ast$\\ \hline
$\ast$ & $0$ 
\end{tabular}$, which rules out one subcase. Similarly, it is not possible for 
$\sigma \in 
\begin{tabular}{l|l}
$1$ & $\ast$\\ \hline
$\ast$ & $\ast$ 
\end{tabular}$, 
$\sigma \not \in 
\begin{tabular}{l|l}
$1$ & $1$\\ \hline
$1$ & $1$ 
\end{tabular}$, 
and 
$\sigma \not \in 
\begin{tabular}{l|l}
$1$ & $\ast$\\ \hline
$\ast$ & $0$ 
\end{tabular}$. So we are left with two possible subcases, both of which lead to no unpaired faces. 

\underline{\textit{Subcase 1:}} Suppose that 
$\sigma \in 
\begin{tabular}{l|l}
$2$ & $\ast$\\ \hline
$\ast$ & $\ast$ 
\end{tabular}$, 
$\sigma \in 
\begin{tabular}{l|l}
$1$ & $\ast$\\ \hline
$\ast$ & $\ast$ 
\end{tabular}$
$\sigma \not \in 
\begin{tabular}{l|l}
$2$ & $2$\\ \hline
$2$ & $2$ 
\end{tabular}$, 
$\sigma \not \in 
\begin{tabular}{l|l}
$2$ & $\ast$\\ \hline
$\ast$ & $1$ 
\end{tabular}$,
and 
$\sigma \not \in 
\begin{tabular}{l|l}
$1$ & $\ast$\\ \hline
$\ast$ & $0$ 
\end{tabular}$. 
Then, it follows that $\sigma \in 
\begin{tabular}{l|l}
$2$ & $\ast$\\ \hline
$\ast$ & $0$ 
\end{tabular}$
and 
$\sigma \in 
\begin{tabular}{l|l}
$1$ & $1$\\ \hline
$1$ & $1$ 
\end{tabular}$. We now notice that the intersection of the perfect matchings of 
$\begin{tabular}{l|l}
$1$ & $1$\\ \hline
$1$ & $1$ 
\end{tabular}$
and 
$\begin{tabular}{l|l}
$2$ & $\ast$\\ \hline
$\ast$ & $0$ 
\end{tabular}$
is contained in a perfect matching of
$\begin{tabular}{l|l}
$2$ & $\ast$\\ \hline
$\ast$ & $1$ 
\end{tabular}$, but by assumption 
$\sigma \not \in 
\begin{tabular}{l|l}
$2$ & $\ast$\\ \hline
$\ast$ & $1$ 
\end{tabular}$ so there are no unpaired faces. 

\underline{\textit{Subcase 2:}}
Suppose now that 
$\sigma \in 
\begin{tabular}{l|l}
$2$ & $\ast$\\ \hline
$\ast$ & $\ast$ 
\end{tabular}$, 
$\sigma \in 
\begin{tabular}{l|l}
$1$ & $\ast$\\ \hline
$\ast$ & $\ast$ 
\end{tabular}$,
$\sigma \not \in 
\begin{tabular}{l|l}
$2$ & $2$\\ \hline
$2$ & $2$ 
\end{tabular}$, 
$\sigma \not \in 
\begin{tabular}{l|l}
$1$ & $1$\\ \hline
$1$ & $1$ 
\end{tabular}$, and
$\sigma \not \in 
\begin{tabular}{l|l}
$2$ & $\ast$\\ \hline
$\ast$ & $0$ 
\end{tabular}$.
Then it follows that 
$\sigma \in 
\begin{tabular}{l|l}
$2$ & $\ast$\\ \hline
$\ast$ & $1$ 
\end{tabular}$
and 
$\sigma \in 
\begin{tabular}{l|l}
$1$ & $\ast$\\ \hline
$\ast$ & $0$. 
\end{tabular}$
The intersection of the perfect matchings of 
$\begin{tabular}{l|l}
$2$ & $\ast$\\ \hline
$\ast$ & $1$ 
\end{tabular}$ and 
$\begin{tabular}{l|l}
$1$ & $\ast$\\ \hline
$\ast$ & $0$ 
\end{tabular}$ 
is contained in the perfect matching 
$\begin{tabular}{l|l}
$1$ & $1$\\ \hline
$1$ & $1$ 
\end{tabular}$. Since this is a contradiction to the assumption that $\sigma \not \in 
\begin{tabular}{l|l}
$1$ & $1$\\ \hline
$1$ & $1$ 
\end{tabular}$, there are no unpaired faces left from this case. 

\textbf{Case 3.2: Suppose $\theta$ is of type $(t2.1)$ and that $\delta \in \theta$.} That is, $\theta = \sigma \cup \delta \cup 
\beta$ where $\alpha, \beta, \gamma, \delta \not \in \sigma$ and 
$\sigma \cup \delta \in 
\begin{tabular}{l|l}
$2$ & $\ast$\\ \hline
$\ast$ & $\ast$ 
\end{tabular}$, 
$ \sigma \cup \delta  \in 
\begin{tabular}{l|l}
$1$ & $\ast$\\ \hline
$\ast$ & $\ast$ 
\end{tabular}$, 
$\sigma \cup \delta \not \in
\begin{tabular}{l|l}
$2$ & $2$\\ \hline
$2$ & $2$ 
\end{tabular}$, and 
($\sigma \cup \delta \not \in 
\begin{tabular}{l|l}
$2$ & $\ast$\\ \hline
$\ast$ & $1$ 
\end{tabular}$ 
or 
$\sigma \cup \delta \not \in 
\begin{tabular}{l|l}
$1$ & $1$\\ \hline
$1$ & $1$ 
\end{tabular}$).
If $\theta\setminus\delta=\sigma\cup \beta$ is also of type $(t2.1)$, then 
$\theta$ and $\theta \setminus\delta$ are paired using $\delta$.
So assume 
$\sigma\cup \beta$ is not of type $(t2.1)$.
Therefore, 
 $\sigma \not \in 
\begin{tabular}{l|l}
$2$ & $\ast$\\ \hline
$\ast$ & $\ast$ 
\end{tabular}$,
or
$ \sigma \not \in 
\begin{tabular}{l|l}
$1$ & $\ast$\\ \hline
$\ast$ & $\ast$ 
\end{tabular}$, 
or
$\sigma \in 
\begin{tabular}{l|l}
$2$ & $2$\\ \hline
$2$ & $2$ 
\end{tabular}$,
or
($\sigma \in 
\begin{tabular}{l|l}
$2$ & $\ast$\\ \hline
$\ast$ & $1$ 
\end{tabular}$ 
and 
$\sigma \in 
\begin{tabular}{l|l}
$1$ & $1$\\ \hline
$1$ & $1$ 
\end{tabular}$). Notice that it is not possible for 
$\sigma \not \in 
\begin{tabular}{l|l}
$2$ & $\ast$\\ \hline
$\ast$ & $\ast$ 
\end{tabular}$
since $\sigma \cup \delta \in 
\begin{tabular}{l|l}
$2$ & $\ast$\\ \hline
$\ast$ & $\ast$ 
\end{tabular}$ and, similarly, it is not possible for 
$\sigma \not \in 
\begin{tabular}{l|l}
$1$ & $\ast$\\ \hline
$\ast$ & $\ast$ 
\end{tabular}$
since 
$\sigma \cup \delta 
\in 
\begin{tabular}{l|l}
$1$ & $\ast$\\ \hline
$\ast$ & $\ast$ 
\end{tabular}$. Hence, the faces left unpaired are such that 
$\sigma \in 
\begin{tabular}{l|l}
$2$ & $\ast$\\ \hline
$\ast$ & $0$ 
\end{tabular}$, 
and 
$\sigma \in 
\begin{tabular}{l|l}
$1$ & $\ast$\\ \hline
$\ast$ & $0$ 
\end{tabular}$, and 
($\sigma \in 
\begin{tabular}{l|l}
$2$ & $2$\\ \hline
$2$ & $2$ 
\end{tabular}$ or 
($\sigma \in 
\begin{tabular}{l|l}
$2$ & $\ast$\\ \hline
$\ast$ & $1$ 
\end{tabular}$
and 
$\sigma \in 
\begin{tabular}{l|l}
$1$ & $1$\\ \hline
$1$ & $1$ 
\end{tabular}$)). There are two subcases to consider: 

\underline{\textit{Subcase 1:}} Suppose first that $\sigma \in 
\begin{tabular}{l|l}
$2$ & $2$\\ \hline
$2$ & $2$ 
\end{tabular}$. 
Then, $\sigma$ is in the intersection between the perfect matchings of 
$\begin{tabular}{l|l}
$2$ & $\ast$\\ \hline
$\ast$ & $0$ 
\end{tabular}$, 
$\begin{tabular}{l|l}
$1$ & $\ast$\\ \hline
$\ast$ & $0$ 
\end{tabular}$, and 
$\begin{tabular}{l|l}
$2$ & $2$\\ \hline
$2$ & $2$ 
\end{tabular}$. 
Therefore the unpaired faces are $\sigma \cup \beta \cup \delta$ where\\ 
$\sigma \in \{\emptyset, \{c_{0,0}\}, \{c_{1,1}\}, \{b_{3,2}\}, \{b_{2,2}\}, \{c_{0,0},b_{3,2}\}, \{c_{0,0}, b_{2,2}\}, \{c_{1,1}, b_{3,2}\}, \{c_{1,1}, b_{2,2}\} \}$.

\underline{\textit{Subcase 2:}} Suppose now that 
$\sigma \in 
\begin{tabular}{l|l}
$2$ & $\ast$\\ \hline
$\ast$ & $1$ 
\end{tabular}$
and 
$\sigma \in 
\begin{tabular}{l|l}
$1$ & $1$\\ \hline
$1$ & $1$ 
\end{tabular}$. The intersection of 
$\begin{tabular}{l|l}
$2$ & $\ast$\\ \hline
$\ast$ & $0$ 
\end{tabular}$, 
$\begin{tabular}{l|l}
$1$ & $\ast$\\ \hline
$\ast$ & $0$ 
\end{tabular}$, 
$\begin{tabular}{l|l}
$2$ & $\ast$\\ \hline
$\ast$ & $1$ 
\end{tabular}$, and 
$\begin{tabular}{l|l}
$1$ & $1$\\ \hline
$1$ & $1$ 
\end{tabular}$ is 
$\mathcal{P}(b_{0,2}, a_{1,1}, c_{0,0}, c_{3,3}, a_{3,2}, b_{3,2})$.

Therefore, the unpaired faces obtained from Case 3.2 are those of the form $\sigma \cup \beta \cup \delta$ where 
 \(\sigma \in 
\mathcal{P}(b_{0,2}, a_{1,1}, c_{0,0}, c_{3,3}, a_{3,2}, b_{3,2}) \cup 
\{\{c_{1,1}\}, \{b_{2,2}\}, \{c_{0,0},b_{2,2}\}, \) 
\( \{c_{1,1}, b_{3,2}\}, \{c_{1,1}, b_{2,2}\} \}.\)

We are now left with three types of unpaired faces: 
\begin{enumerate}
    \item[(1)] $\sigma \cup \beta \cup \gamma$ where $\sigma \in \mathcal{P}(c_{0,0},b_{3,2},c_{1,0}, b_{2,1})$ (Case 4.2)
    \item[(2)] $\sigma \cup 
\gamma \cup \delta$ where\\
$\sigma \in \mathcal{P}(b_{0,2}, b_{1,3}, c_{3,3}, c_{2,3}) \cup \{ \{b_{1,2}\}, \{c_{2,2}\}, \{b_{1,2},c_{2,2}\}, \{b_{1,2},c_{3,3}\}, \{b_{0,2},c_{2,2}\} \}$ (Case 2.2)
\item[(3)] $\sigma \cup \beta \cup \delta$ where\\ 
$ \sigma \in \mathcal{P}(b_{0,2}, a_{1,1}, c_{0,0}, c_{3,3}, a_{3,2}, b_{3,2}) \cup 
\{\{c_{1,1}\}, \{b_{2,2}\}, \{c_{0,0},b_{2,2}\}, 
 \{c_{1,1}, b_{3,2}\}, \{c_{1,1}, b_{2,2}\} \}.$ 

(Case 3.2)

\end{enumerate}

Pairing with $c_{1,0}$ matches all faces in (1). Then, pairing with $b_{0,2}$ leaves the faces: \begin{enumerate}
    \item[(2)] $\sigma \cup 
\gamma \cup \delta$ where $\sigma \in \{ \{b_{1,2}\}, \{b_{1,2},c_{2,2}\}, 
\{b_{1,2},c_{3,3}\} \}$
\item[(3)] $\sigma \cup \beta \cup \delta$ where 
$\sigma \in 
\{\{c_{1,1}\}, \{b_{2,2}\},  \{c_{0,0}, b_{2,2}\}, \{c_{1,1}, b_{3,2}\}, \{c_{1,1}, b_{2,2}\} \}.$

\end{enumerate}

Finally we perform a small series of element pairings. 
The element pairing using $c_{3,3}$ pairs $\{b_{1,2} \}$ with 
$\{b_{1,2}, c_{3,3}\}$, 
 an element pairing using $c_{0,0}$ pairs $\{b_{2,2}\}$ with $\{c_{0,0},b_{2,2}\}$, and an element pairing using $b_{3,2}$ pairs $c_{1,1}$ with $\{c_{1,1}, b_{3,2} \}$. This leaves us with two critical cells $\{b_{1,2}, c_{2,2}\} \cup \delta \cup \gamma$ and $\{c_{1,1},b_{2,2}\} \cup \beta \cup \delta$ and the homotopy type $S^3 \vee S^3$.  
\end{proof}

\section{Conclusion and further directions}
Throughout this paper, our main guiding question has been, are the perfect matching complexes of honeycomb graphs all contractible or homotopy equivalent to a wedge of spheres? We have considered  the homotopy type of the perfect matching complexes of honeycomb graphs $H_{\ell \times m\times n}$.  When $\ell = 1$, we were able to compute all homtopy types, but for $\ell\ge 2$, the only homotopy type we were able to compute is for $H_{2\times 2\times 2}$.  It appears, at this time, that for larger honeycomb graphs, we need a new strategy. 

%
%
%
%
%

Our motivation for this project has been to better understand (ordinary) 
matching complexes of honeycomb graphs, $\match (H_{k \times m \times n})$.
These complexes include faces corresponding to matchings that are not contained in any perfect matching; thus
they are more complicated than perfect matching complexes.
For example, Matsushita  (\cite{matsushita2019}) showed that the matching 
complex of the honeycomb graphs $H_{1,1,n}$ has the homotopy type of a wedge of
spheres, as part of a more general result on polygonal line tilings.  
This contrasts with our result on these graphs: if you consider only the 
subcomplex of perfect matchings, the complex is contractible (for $n\ge 2$).

We are interested more generally in the relationship between the (ordinary)
matching complex and the perfect matching complex.  
Our curiosity in (ordinary) matching complexes has not diminished.
Are the matching complexes of honeycomb graphs contractible or homotopy 
equivalent to wedges of spheres? What can we say about the (ordinary or perfect)
matching complexes of more general classes of graphs, such as bipartite graphs?

\appendix

\section{Perfect matchings of $H_{2\times 2\times 2}$}



\bibliography{refs}
\end{document}